\documentclass[11pt,reqno]{amsart}
\usepackage{amssymb,mathrsfs,mathtools,graphicx,enumerate,color,colortbl}
\usepackage{hyperref}
\usepackage{comment}
\title[Contraction of viscous-dispersive shock for NSK]{Contraction of viscous-dispersive shocks: \\Zero viscosity-capillarity limits}
\author[Eun]{Namhyun Eun}
\address[Namhyun Eun]
{ Department of Mathematical Sciences, \newline
Korea Advanced Institute of
Science and Technology \\
Daejeon 34141, Korea}
\email{namhyuneun@kaist.ac.kr}

\author[Kang]{Moon-Jin Kang}
\address[Moon-Jin Kang]
{ Department of Mathematical Sciences, \newline
Korea Advanced Institute of
Science and Technology \\
Daejeon 34141, Korea}
\email{moonjinkang@kaist.ac.kr}

\author[Kim]{Jeongho Kim}
\address[Jeongho Kim]
{ Department of Applied Mathematics, \newline
Kyung Hee University \\
1732 Deogyeong-daero, Giheung-gu, Yongin-si, Gyeonggi-do 17104, Korea}
\email{jeonghokim@khu.ac.kr}
\newtheorem{theorem}{Theorem}[section]
\newtheorem{lemma}{Lemma}[section]

\newtheorem{remark}{Remark}[section]

\numberwithin{figure}{section}

\newcommand{\beq}{\begin{equation}}
\newcommand{\eeq}{\end{equation}}
\newcommand{\bsp}{\begin{split}}
\newcommand{\esp}{\end{split}}

\newcommand{\wc}{\rightharpoonup}

\newtheorem{theo}{Theorem}[section]
\newtheorem{prop}[theo]{Proposition}

\newcommand{\Bcal}{\mathcal{B}}

\newcommand{\Dcal}{\mathcal{D}}
\newcommand{\Ecal}{\mathcal{E}}
\newcommand{\Fcal}{\mathcal{F}}
\newcommand{\Gcal}{\mathcal{G}}
\newcommand{\Hcal}{\mathcal{H}}
\newcommand{\Ical}{\mathcal{I}}
\newcommand{\Jcal}{\mathcal{J}}

\newcommand{\Mcal}{\mathcal{M}}

\newcommand{\Pcal}{\mathcal{P}}

\newcommand{\Rcal}{\mathcal{R}}

\newcommand{\Xcal}{\mathcal{X}}

\newcommand{\VBar}{\overline{V}}

\newcommand{\Ubar}{\bar{U}}

\renewcommand{\hbar}{\bar{h}}

\newcommand{\ubar}{\bar{u}}
\newcommand{\vbar}{\bar{v}}

\newcommand{\Util}{\tilde{U}}

\newcommand{\htil}{\tilde{h}}

\newcommand{\util}{\tilde{u}}
\newcommand{\vtil}{\tilde{v}}

\newcommand{\vt}{{\tilde{v}}}
\newcommand{\ut}{{\tilde{u}}}

\renewcommand{\a}{\alpha}
\renewcommand{\b}{\beta}
\newcommand{\g}{\gamma}
\renewcommand{\d}{\delta}

\newcommand{\z}{\zeta}
\newcommand{\et}{\eta}

\renewcommand{\k}{\kappa}
\renewcommand{\l}{\lambda}

\newcommand{\m}{\mu}
\newcommand{\n}{\nu}
\newcommand{\x}{\xi}

\renewcommand{\r}{\rho}
\newcommand{\s}{\sigma}
\renewcommand{\t}{\tau}

\renewcommand{\O}{\Omega}

\newcommand{\e}{\varepsilon}
\newcommand{\ee}{\epsilon}

\newcommand{\RR}{{\mathbb R}}

\newcommand{\abs}[1]{\left|#1\right|}

\newcommand{\norm}[1]{\left\|#1\right\|}

\newcommand{\one}[1]{\mathbf{1}_{\{#1\}}}
\newcommand{\oone}[1]{\mathbf{1}_{#1}}

\newcommand{\bpf}[1]{\noindent\textbf{Proof of \eqref{#1}:}}

\newcommand{\step}[1]{\vskip0.2cm \noindent{\it Step #1:} }

\newcommand{\vtn}{{\tilde{v}^\nu}}
\newcommand{\utn}{{\tilde{u}^\nu}}

\newcommand{\vn}{{v^\nu}}
\newcommand{\un}{{u^\nu}}

\definecolor{black}{rgb}{0.0, 0.0, 0.0}
\definecolor{red}{rgb}{1.0, 0.5, 0.5}

\newcommand{\rd}{\partial}

\newcommand{\That}{\hat{T}}

\begin{document}
\bibliographystyle{plain}

\date{\today}

\subjclass{35B35,76N10,76N15} \keywords{Navier--Stokes--Korteweg system, Viscous-dispersive shock, Contraction, Existence, Stability, Zero viscosity-capillarity limit, Barotropic Euler system, Riemann shock}

\thanks{\textbf{Acknowledgment.} N. Eun and M.-J. Kang were supported by Samsung Science and Technology Foundation under Project Number SSTF-BA2102-01. J. Kim was supported by Samsung Science and Technology Foundation under Project Number SSTF-BA2401-01.}

\begin{abstract}
We prove the contraction property of any large solution perturbed from a viscous-dispersive shock wave of the Navier--Stokes--Korteweg (NSK) system. 
The contraction holds up to a dynamical shift, since the contraction is measured by the relative entropy that is locally $L^2$.
We use the contraction property to show the global existence of large solution perturbed from a viscous-dispersive shock wave.
To prove the contraction property, we first employ the effective velocity to transform the NSK system into the system of two degenerate parabolic equations, then apply the method of $a$-contraction with shifts. 
The contraction property does not depend on the strengths of viscosity and capillarity. Based on this uniformity, we show the existence of zero viscosity-capillarity limits of solutions to the NSK system, on which Riemann shocks are unique and stable up to shifts.
\end{abstract}
\maketitle \centerline{\date}

\tableofcontents

\section{Introduction}
\setcounter{equation}{0}
We consider an one-dimensional compressible fluid of Korteweg type, whose dynamics is governed by the Navier--Stokes--Korteweg (NSK) system which, in the Lagrangian mass coordinates, reads as follows:
\begin{equation} \label{main0}
\begin{cases}
v_t - u_x = 0, \\
u_t + p(v)_x = \left(\mu(v)\frac{u_x}{v}\right)_x + \left(\kappa(v)\left(-\frac{v_{xx}}{(v)^5}+\frac{5(v_x)^2}{2(v)^6}\right)-\frac{\kappa'(v)(v_x)^2}{2(v)^5}\right)_x.
\end{cases}
\end{equation}
Here, $v$ and $u$ denote specific volume and fluid velocity and $p(v)$, $\m(v)$, and $\kappa(v)$ respectively represent the pressure, the viscosity coefficient and the capillarity coefficient of the fluid.
We focus on a case of polytropic gas, and viscosity coefficients that may degenerate near vacuum, that is
\begin{equation} \label{pressure}
p(v)=v^{-\g}, \qquad \m(v)= bv^{-\a}, 
\end{equation}
with the exponent \(\g\ge1\). The constants \(\a\ge0\) and $b>0$ depend only on the type of gas.
We assume that the capillarity coefficient $\kappa(v)$ has a close relationship with the viscosity $\mu(v)$, and it is given by
\begin{equation}\label{kappa}
	\kappa(v)=c\mu(v)^2 v^3,
\end{equation}
for some constant $c>0$.
This specific relation between the viscosity and capillarity is naturally assumed in various literature and we refer to for instance, \cite{BGLV-ARMA19,CH-SIMA22,CH-SIMA13,GL-CPAM15}.

\vspace{2mm}
The system \eqref{main0} appears in many fields of application, mainly used for describing hydrodynamic phenomena of the quantum mechanical system (for $\kappa(v) = v$), or modeling the diffuse interface for the two-phase fluid.
In particular, the Euler--Korteweg system can be derived from the nonlinear Schr\"odinger equation via the Madelung transformation (see \cite{CDS-12, Madelung1927}).
The following form of nonlinear Schr\"odinger equation 
\[
i \ee \psi^{\ee}_t + \frac{\ee^2}{2} \psi^{\ee}_{xx} = f(\abs{\psi^{\ee}}^2) \psi^{\ee}
\]
and the Madelung transform 
\[
\psi^\ee = \sqrt{\r(t,x)} e^{i \phi(t,x)/{\ee}}
\]
yield the isentropic Euler system (in the Eulerian coordinates) in terms of \((\rho,\phi_x)\), and the Korteweg tensor for \(\k(v)=v(=1/\rho)\) (which implies \(\m(v)=1/v\)) appears on the right-hand side.
Thus, the NSK system, equipped with \(\k(v)=v\), can be viewed as a dissipative extension of quantum fluid models.
Such models arise in the description of various physical contexts, including quantum semiconductors, Bose gases, and superfluids (see \cite{ref-quantumsemiconductor, ref-bosegases, ref-superfluid}).

\vspace{2mm}
The original idea of fluid with capillarity dates back to around early of 20th century, which was introduced by van der Waals \cite{vanderWaals-ZPC94} and Korteweg \cite{Korteweg-01}.
Later, the NSK system \eqref{main0} was rigorously derived by Dunn and Serrin \cite{DS-ARMA85} via thermomechanical theories of interstitial working.
After its introduction, the NSK system has been extensively studied.
To name a few, Hattori and Li \cite{HL-SIMA94,HL-JMAA96} initiates a study on the existence of strong solutions to the NSK system, and many other results on a weak solution, or solutions with critical regularity have been investigated.
We refer to \cite{BDL-CPDE03,CCDZ-JDE15,DD-Poincare01,Haspot-Pascal09,Haspot-JMFM11,Kotschote-Poincare08} and references therein for the existence results for the NSK system. \\

Our primary goal is to prove the contraction property (up to shifts) of large solutions of \eqref{main0} around a  viscous-dispersive shock profile as a solution to
\begin{equation} \label{shock_0}
\begin{cases}
-\s \vtil' - \util' = 0, \\
-\s \util' + p(\vtil)' = \Big(\mu(\vtil)\frac{\util'}{\vtil}\Big)' 
+ \Big(\kappa(\vtil)\Big(-\frac{\vtil''}{\vtil^5}+\frac{5(\vtil')^2}{2\vtil^6}\Big)-\frac{\kappa'(\vtil)(\vtil')^2}{2\vtil^5}\Big)', \\
\lim_{\x\to\pm\infty} (\vtil^\n,\util^\n)(\x)=(v_\pm,u_\pm),
\end{cases}
\end{equation}
where the constant states $(v_\pm,u_\pm)$ satisfy the Rankine-Hugoniot condition and Lax entropy condition:
\begin{align}
\begin{aligned} \label{end-con}
&\exists~\s \quad \text{s.t.}~\left\{
\begin{aligned}
&-\s(v_+-v_-) -(u_+-u_-) =0, \\
&-\s(u_+-u_-) +p(v_+)-p(v_-) =0,
\end{aligned} \right. \\
&\text{and either \(v_->v_+\) and \(u_->u_+\) or \(v_-<v_+\) and \(u_->u_+\) holds.}
\end{aligned}
\end{align}
If \(v_->v_+\), \((\vtil,\util)\) is a \(1\)-shock with velocity  \(\s=-\sqrt{-\frac{p(v_+)-p(v_-)}{v_+-v_-}}<0\), and if \(v_-<v_+\), \((\vtil,\util)\) is a \(2\)-shock with velocity \(\s=\sqrt{-\frac{p(v_+)-p(v_-)}{v_+-v_-}}>0\).

The contraction holds up to a dynamical shift, since the contraction is measured by the relative entropy that is locally $L^2$.
For the contraction with regular dynamical shifts, we need to prove the global existence of strong solutions evolving from regular initial perturbation without smallness assumption. 
The contraction property does not depend on the strengths of viscosity and capillarity. So, this uniformity may guarantee the existence of zero viscosity-capillarity limits of solutions to the scaled NSK system:
\begin{equation} \label{inveq}
\begin{cases}
v_t^\nu - u_x^\nu = 0, \\
u_t^\nu + p(v^\nu)_x = \nu\left(\mu(v^\nu)\frac{u^\nu_x}{v}\right)_x + \nu^2\left(\kappa(v^\nu)\left(-\frac{v^\nu_{xx}}{(v^\nu)^5}+\frac{5(v^\nu_x)^2}{2(v^\nu)^6}\right)-\frac{\kappa'(v^\nu)(v^\nu_x)^2}{2(v^\nu)^5}\right)_x,
\end{cases}
\end{equation}
where $\nu>0$ is a natural vanishing parameter.
Regarding the scale of vanishing parameters, i.e., \(\n\) and \(\n^2\), we refer to \cite{CH-SIMA13,GL-CPAM15} for vanishing viscosity-capillarity limits and \cite{BS-PRSE85,Pego-TAMS85} for monotone profiles of viscous-dispersive shock waves.
Indeed, we use the contraction property to obtain the existence of zero viscosity-capillarity limits of solutions to \eqref{inveq}, on which the Riemann shock \eqref{shock-0}  is unique and stable (up to shifts). The (Riemann) shock of the form: 
\begin{equation} \label{shock-0}
(\vbar,\ubar)(x-\s t)= \left\{
\begin{aligned}
(v_-,u_-) \quad \text{if } x-\s t<0, \\
(v_+,u_+) \quad \text{if } x-\s t>0,
\end{aligned} \right.
\end{equation}
is a self-similar entropy solution to the polytropic Euler system with $p(v)=v^{-\g}, \g\ge1$:
\begin{equation} \label{Euler}
\begin{cases}
v_t - u_x = 0, \\
u_t + p(v)_x = 0.
\end{cases}
\end{equation}

There has been a few literature considering {\it zero viscosity-capillarity limit} of the NSK system \eqref{inveq} towards the Euler system \eqref{Euler}, when both the viscosity and capillary parameters vanish simultaneously. Charve and Haspot \cite{CH-SIMA13} showed that the solution to the NSK system converges to an entropy solution of the Euler system, with a special choice of viscosity and capillarity coefficients. This was generalized to broader class of viscosity and capillarity by Germain and LeFloch \cite{GL-CPAM15}. We also refer to \cite{CH-SIMA22} for the vanishing capillary limit of the NSK system. However, to the best of our knowledge, there has been no result on the stability of entropy solutions obtained by   zero viscosity-capillarity limit. Therefore, we establish the first result on the uniqueness and stability of shock \eqref{shock-0} in the class of vanishing viscosity-capillarity limit from the NSK system \eqref{inveq} as in Theorem \ref{thm_inviscid}. This result is achieved by Theorem \ref{thm_main} on the contraction property. This property may provide a key cornerstone even for stability of more general entropy solutions. In the study \cite{CKV-24} on inviscid limits from Navier--Stokes system (i.e. $\kappa\equiv0$) to Euler system, the authors proved the stability of  entropy solutions to the isentropic Euler evolving from  small
BV initial data, in the class of  inviscid limits of solutions to the associated Navier--Stokes system. There, the results \cite{KV-JEMS21,KV-Inven21,KV-JDE22} on contraction property of shocks were importantly used. We leave the study on the zero viscosity-capillarity limit of NSK to stable BV solutions for future work.

\subsection{Main results}
We now present our main results in the following order: the contraction property; the existence and uniqueness of strong large perturbations; and  the uniqueness and stability of Riemann shock to the polytropic Euler system in the class of zero viscosity-capillarity limits from the NSK system.

We first introduce the following relative functional \(E(\cdot|\cdot)\) to measure the contraction property: for any positive functions \(v_1,v_2,\) and any functions \(u_1,u_2,\)
\[
E((v_1,u_1)|(v_2,u_2)) \coloneqq 
\frac{1}{2}\Big(u_1-\frac{\t_1 \g (v_1)_\x}{(v_1)^{\a+1}}-u_2+\frac{\t_1 \g (v_2)_\x}{(v_2)^{\a+1}}\Big)^2
+Q(v_1|v_2).
\]
Here, \(\t_1:=\frac{1+\sqrt{1-4c}}{2}\) for the constant $0\le c\le \frac{1}{4}$ from the relation \eqref{kappa}.
Let us denote \(h \coloneqq u- \frac{\t_1\g v_\x}{v^{\a+1}}\) (resp. \(\htil \coloneqq \util - \frac{\t_1\g \vtil'}{\vtil^{\a+1}}\)) as an effective velocity which is introduced in a similar form by \cite{CH-SIMA13,GL-CPAM15}, motivated by 
the Bresch--Desjardins entropy (see for instance, \cite{BD-CMP03,BD-JMPA06,BDL-CPDE03}).
Notice that the above functional is  the relative entropy associated to the entropy of \eqref{Euler}:
\begin{equation} \label{et-def}
\et((v_1,h_1)|(v_2,h_2)) \coloneqq \frac{1}{2}(h_1-h_2)^2 + Q(v_1|v_2),
\end{equation}
where
\[
Q(v_1|v_2)
\coloneqq Q(v_1)-Q(v_2)-Q'(v_2)(v_1-v_2),
\]
associated with the strictly convex function 
\begin{equation*}
\begin{cases}
Q(v) \coloneqq \frac{v^{-\g+1}}{\g-1} \qquad \text{ for } \g>1,\\
Q(v) \coloneqq -\log v \qquad \text{ for } \g=1,
\end{cases} \quad v>0.
\end{equation*}
For the contraction property with dynamical shifts, we consider the following function class:
\begin{multline*}
\Xcal_T \coloneqq \{(v,u) \mid v-\underline{v} \in C([0,T];H^2(\RR))\cap L^2(0,T;H^3(\RR)),
0<v^{-1}\in L^\infty((0,T)\times\RR)\\
u-\underline{u}\in C([0,T];H^1(\RR))\cap L^2(0,T;H^2(\RR))\},
\end{multline*}
where \(\underline{v}\) and \(\underline{u}\) are smooth monotone functions such that
\begin{equation} \label{sm-end}
\underline{v}(x)=v_{\pm} \quad \text{ and } \quad
\underline{u}(x)=u_{\pm} \quad \text{ for } \pm x \ge 1.
\end{equation}
The contraction result applies to any NSK solutions belonging to \(\Xcal_T\), and we note that the existence of solutions within this function space is obtained in our second theorem.

\begin{theorem} \label{thm_main}
Let \(\g,\a\) be any constants satisfying 
\begin{equation} \label{g-a-condition}
1 \le \g \le \frac{5}{4}, \qquad 0 \le \a \le \g \le 1+\frac{\a}{3}.
\end{equation}
Assume $c\in[0,\frac{9}{100}]$ for the constant in \eqref{kappa}. Then, for a given constant state \((v_-,u_-) \in \RR^+\times\RR\), there exist constants \(\e_0,\d_0>0\) such that the following holds.\\
Fix any \(\e\in(0,\e_0),\) \(\l\in(\d_0^{-1}\e,\d_0),\) and \((v_+,u_+)\in\RR^+\times\RR\) satisfying \eqref{end-con} with \(|p(v_+)-p(v_-)|=\e\). Let \(\Util\coloneqq(\vtil,\util)\) be the viscous-dispersive shock defined in \eqref{shock_0} and let \(U\coloneqq(v,u)\) be a solution to \eqref{main0} in \(\Xcal_T\) for some $T>0$ subject to an initial datum \(U_0\coloneqq(v_0,u_0)\) which satisfies \(\int_\RR E(U_0|\Util)dx < \infty\). Then, there exist a smooth monotone function \(a\colon\RR\to\RR^+\) with \(\lim_{x\to\pm\infty}a(x)=1+a_\pm\) for some constants \(a_-\) and \(a_+\) with \(|a_+-a_-|=\l\), a shift function \(X\in W^{1,1}(0,T)\) with \(X(0)=0\) and a constant \(C>0\) (independent of \(\e,\l,\d_0\) and \(T\)) such that \begin{equation} \label{cont_main-u}
	\begin{aligned}
		&\int_\RR a(\x)E(U(t,\x+X(t))|\Util(\x))d\x \\
		&\qquad
		+ \d_0\frac{\e}{\l}\int_0^T \int_\RR \abs{\s_\e a'(\x)} Q(v(s,\x+X(s))|\vtil(\x))d\x ds \\
		&\qquad
		+ \d_0 \t_1 \int_0^T \int_\RR a(\x) v^{\g-\a}(s,\x+X(s))\abs{\rd_\x \left(p(v(s,\x+X(s)))-p(\vtil(\x))\right)}^2 d\x ds \\
		&\qquad
		+ \d_0 \t_2\g \int_0^T \int_\RR a(\x)  v^{-\a-1}(s,\x+X(s)) |\rd_\x (h(s,\x+X(s))-\htil(\x))|^2 d\x ds\\
		&\le \int_\RR a(\x)E(U_0(\x)|\Util(\x))d\x,
	\end{aligned}
\end{equation}
and
\begin{equation} \label{shift-control-u}
	\begin{aligned}
		&|\dot{X}(t)| \le \frac{C}{\e^2}(f(t)+1), \quad \text{for a.e. } t\in[0,T], \\
		&\text{for some positive function } f \text{ satisfying } \norm{f}_{L^1(0,T)} \le \frac{\l}{\d_0\e}\int_\RR E(U_0(x)|\Util(x))dx.
	\end{aligned}
\end{equation}

\end{theorem}

\begin{remark}
(1)	In Theorem \ref{thm_main}, we restrict the range of the parameter $c$ as $[0,\frac{9}{100}]$. In this range of $c$, we have $\tau_1=\frac{1+\sqrt{1-4c}}{2}\ge\frac{9}{10}$, which will be importantly used in the proof of the contraction estimate. Notice that, the extreme case of $c=0$ (i.e. $\kappa=0$) follows from the contraction property of the NS system as in \cite{EEKO-ISO,KV-JEMS21,KV-Inven21}. \\
(2) The result of Theorem \ref{thm_main} is independent of the strength of the constant $b$ of $\mu$ in \eqref{pressure}. 
\end{remark}

\vspace{2mm}

Next, we state the global-in-time existence result for solutions to the NSK system in the function class \(\Xcal_T\), which accommodates arbitrarily large initial data.

\begin{theorem} \label{thm:existence}
Let \(\g,\a\) be any constants satisfying \eqref{g-a-condition} and \(\a<1/2\), and \(c\in[0,\frac{9}{100}]\).
Let \((v_-,u_-)\in\RR^+\times\RR\) be a given constant state.
Let \(\e_0>0\) be the constant provided by Theorem \ref{thm_main} associated with \((v_-,u_-)\).
For any \((v_+,u_+)\in\RR^+\times\RR\) which satisfies \eqref{end-con} and \(|p(v_+)-p(v_-)|=\e<\e_0\), let \(\Util=(\vtil,\util)\) be the traveling wave solution of \eqref{shock_0} connecting \((v_-,u_-)\) and \((v_+,u_+)\).
Consider an initial datum \(U_0 \coloneqq (v_0,u_0)\) with 
\[
v_0-\vtil \in H^2(\RR), \qquad
u_0-\util \in H^1(\RR), \qquad
0<v^{-1} \in L^\infty(\RR).
\]
Then, there exists a unique global-in-time solution \(U\) to \eqref{main0} with \eqref{pressure}--\eqref{kappa} subject to the initial datum $U_0$ such that for any \(T>0\),
\begin{align*}
&(v-\vtil,u-\util) \in C([0,T];H^2(\RR))\cap L^2(0,T;H^3(\RR))
\times C([0,T];H^1(\RR))\cap L^2(0,T;H^2(\RR)),\\
&\hspace{50mm} 0 < v^{-1} \in L^\infty(0,T;L^\infty(\RR)).
\end{align*}
\end{theorem}

\begin{remark}
We refer to the existence result by Burtea--Haspot \cite{BH-22} for the case of $\alpha>1$.
\end{remark}

\vspace{2mm}
To present the stability result of Riemann shock in the class of  zero viscosity-capillarity limits,  we need to extend the notion of relative entropy defined in \eqref{et-def}.
The need for this extension can be demonstrated as follows: we will consider vanishing viscosity-capillarity limits to the NSK system for the first component of \(Q(v_1|v_2)\), i.e., \(v_1\). Since the uniform bounds for solutions are only available in \(L^1\) space, the limits could be measures on \(\RR^+\times\RR\).
This may still be interpreted as a physical phenomenon, namely the emergence of cavitation and this is why we need to generalize the relative entropy to measures defined on \(\RR^+ \times \RR\).
Fortunately, since we only consider Riemann shocks, it suffices to extend the relative entropy only for the cases when we compare a measure \(dv\) with a simple function \(\vbar\) which takes only two values \(v_-\) and \(v_+\).
Let \(v_a\) be the Radon-Nikodym derivative of \(dv\) with respect to the Lebesgue measure and \(dv_s\) be its singular part, i.e., \(dv=v_a dtdx+dv_s\).
Then, we define the relative functional as 
\beq\label{gent}
dQ(v|\vbar) \coloneqq Q(v_a|\vbar)dtdx + \abs{Q'(\VBar(t, x))} dv_s(t,x),
\eeq
where \(\VBar\) is defined as
\[
\VBar(t,x)=
\begin{cases}
\max(v_-,v_+) & \text{for } (t,x)\in \overline{\O_M} (\eqqcolon \text{the closure of }\O_M), \\
\min(v_-,v_+) & \text{for } (t,x)\notin \overline{\O_M},
\end{cases}
\]
and \(\O_M=\{(t,x)|\vbar(t,x)=\max(v_-,v_+)\}\).
Here we use \(\VBar\) instead of \(\vbar\) in order to handle the singular part \(dv_s\).
We remark that the relative entropy is now a measure itself.
Furthermore, if \(v\in L^{\infty}(\RR^+;L^\infty(\RR)+\Mcal(\RR))\) and its Radon-Nikodym derivative \(v_a\) is away from \(0\), then \(dQ(v|\vbar)\) belongs to \(L^{\infty}(\RR^+;\Mcal(\RR))\), where \(\Mcal\) is the space of nonnegative Radon measures.

\vspace{2mm}
We are now ready to present the  the stability and uniqueness of shock to \eqref{Euler}:
\begin{theorem}\label{thm_inviscid}
Let \(\g,\a\) be any constants satisfying \eqref{g-a-condition} and $c\in[0,\frac{9}{100}]$.
Let \((v_-,u_-)\in\RR^+\times\RR\) be a given constant state and \(\e_0>0\) be the constant provided by Theorem \ref{thm_main} associated with \((v_-,u_-)\).
Then, for any \((v_+, u_+)\in \RR^+\times \RR\) satisfying \eqref{end-con} with \(\abs{p(v_+)-p(v_-)} = \e<\e_0\), the following holds.\\
Let \((\vtn, \utn)\) be a viscous-dispersive shock scaled by $(\vtn, \utn)(x-\sigma t) = (\vt, \ut)(\frac{x-\sigma t}{\n}) $ and \(U^0=(v^0, u^0)\) be an initial datum of \eqref{Euler} satisfying 
\[
\Ecal_0 \coloneqq \int_\RR \et((v^0, u^0)|(\vbar, \ubar))dx < \infty.
\]
Let \(\{(v_0^\nu, u_0^\nu)\}_{\nu>0}\) be any sequence of smooth initial data to \eqref{inveq} on \(\RR\) such that
\begin{equation} \label{ini_conv}
	\begin{aligned}
		&\lim_{\n\to 0}v_0^\nu = v^0, \quad \lim_{\n\to 0}u_0^\nu = u^0 \quad a.e., \quad v_0^\nu > 0, \\
		&\lim_{\n\to 0}\int_\RR Q(v_0^\nu|\vtn) + \frac{1}{2}\Big(u_0^\nu-\n \frac{\t_1 \g(v_0^\nu)_x}{(v_0^\nu)^{\a+1}}- \utn+\n\frac{\t_1 \g(\vtn)_x}{(\vtn)^{\a+1}}\Big)^2 dx = \Ecal_0.
	\end{aligned}
\end{equation}
For any given \(T>0\), let \(\{(\vn, \un)\}_{\n>0}\) be a sequence of solutions in \(\Xcal_T\) to \eqref{inveq} with the initial datum \((v_0^\nu, u_0^\nu)\) as above. 
Then, as \(\n\to 0\), there exist limits (up to a subsequence) \(v_\infty\) and \(u_\infty\) such that
\begin{equation}\label{wconv}
	\vn\wc v_\infty, \quad \un\wc u_\infty ~\text{ in }~ \Mcal_{loc}((0, T)\times \RR),
\end{equation}
where \(v_\infty\) lies in \(L^\infty(0, T, L^\infty(\RR)+\Mcal(\RR))\) and \(\Mcal_{loc}((0, T)\times \RR)\) is the space of locally bounded Radon measures. \\
In addition, there exist a shift \(X_\infty\in BV((0, T))\) and a constant \(C>0\) such that \(d Q(v_\infty|\vbar)\in L^\infty(0, T;\Mcal(\RR))\), and for almost every \(t\in (0, T)\), 
\begin{equation} \label{uni-est}
	\int_{x\in\RR} dQ(v_\infty|\vbar(x-X_\infty(\cdot)))(t) 
	+\int_{\RR} \frac{\abs{u_\infty(t,x) - \ubar(x-X_\infty(t))}^2}{2}dx \le C\Ecal_0.
\end{equation}
Moreover, the shift \(X_\infty\) satisfies 
\begin{equation}\label{X-control}
	\abs{X_\infty(t)-\s t} \le \frac{C(T)}{\abs{v_--v_+}}(\sqrt{\Ecal_0}+\Ecal_0).
\end{equation}
Therefore, entropy shocks \eqref{shock-0} with small amplitude of the polytropic Euler system \eqref{Euler} are stable (up to shifts) and unique in the class of vanishing viscosity-capillarity limits of solutions to the NSK system \eqref{inveq}. 

\end{theorem}

\begin{remark}
Following the arguments of the previous works \cite{EEKO-ISO,KV-Inven21}, Theorem \ref{thm_inviscid} is an immediate consequence of Theorem \ref{thm_main}. So, we omit the details of the proof for   \ref{thm_inviscid}.
Indeed, the sequence of well-prepared initial data \(\{(v_0^\nu, u_0^\nu)\}_{\nu>0}\) can be constructed by mollifying an approximate sequence of initial perturbation $\mathcal{E}_0$ as in \cite[Section 5.1]{KV-Inven21}.
Since the contraction estimate \eqref{cont_main-u} still holds for the scalings $U^\nu(t,x):=U(\frac{t}{\nu},\frac{x}{\nu})$ (as the solution to \eqref{inveq}) and $X^\nu(t):=\nu X(\frac{t}{\nu})$, we use the contraction estimate with the well-prepared initial data \eqref{ini_conv} to get the weak compactness for \eqref{wconv}. Notice that $h^\nu \wc u_\infty$ thanks to the uniform bound for the diffusion on $v^\nu$ by the contraction estimate. Then, we obtain the stability estimate \eqref{uni-est}, by using the weakly lower semi-continuity of the generalized relative entropy \eqref{gent}, and the strong convergence of $X^\nu$ by BV compactness from \eqref{shift-control-u}.  Finally, the estimate \eqref{X-control} for the uniqueness result is obtained by using  Rankine-Hugoniot condition with  the linearity of the continuity equation. 
\end{remark}

\vspace{2mm}
The rest of the paper is organized as follows. In Section \ref{sec:trans}, we transform the NSK system \eqref{main0} into the system of two degenerate parabolic equations by introducing the effective velocity $h$, and restate Theorem \ref{thm_main} in terms of $(v,h)$. Then, we present the proof of the restated theorem in Section \ref{sec:3}. In Section \ref{sec:4}, we provide the proof of Proposition \ref{prop:main}, which is a key estimate in Section \ref{sec:3}. Finally, Section \ref{sec:5} is devoted to the proof of Theorem \ref{thm:existence}.

\section{Transformation of the system} \label{sec:trans}
\setcounter{equation}{0}

In this section, we transform the NSK system \eqref{main0} by using an effective velocity  in a similar form as in \cite{CH-SIMA13,GL-CPAM15}, which is motivated by the Bresch--Desjardins entropy \cite{BD-CMP03,BD-JMPA06}.
We define the effective velocity $h$ as
\[h(t,x) = u(t,x) -\tau_1\mu(v)\frac{v_x}{v},\quad \tau_1 = \frac{1+\sqrt{1-4c}}{2},\]
where $c$ is the constant of the relation \eqref{kappa}.
The initial datum \(h_0\) is defined by  \((v_0,u_0)\) in the same manner. Note that the choice of $\tau_1$ satisfies the relation $\tau_1(1-\tau_1)=c$. Then \eqref{main0}$_1$ can be written as
\begin{equation}\label{eq:v_effect}
	v_t - h_x =\tau_1\Big(\mu(v)\frac{v_x}{v}\Big)_x.
\end{equation}
On the other hand, the effective velocity $h$ satisfies
\begin{align}
\begin{aligned}\label{h_eq}
	h_t +p(v)_x &= u_t -\tau_1\Big(\mu(v)\frac{v_x}{v}\Big)_t+p(v)_x\\
	&=(1-\tau_1)\Big(\mu(v)\frac{u_x}{v}\Big)_x+\Big[\kappa(v)\Big(-\frac{v_{xx}}{v^5}+\frac{5v_x^2}{2v^6}\Big)-\frac{\kappa'(v)(v_x)^2}{2v^5}\Big]_x.
\end{aligned}
\end{align}
Using the definition of the effective velocity, the right-hand side of \eqref{h_eq} can be written as
\begin{align*}
	&(1-\tau_1)\Big(\mu(v)\frac{h_x}{v}\Big)_x+\Big[\tau_1(1-\tau_1)\frac{\mu(v)}{v}\Big(\frac{\mu(v)v_x}{v}\Big)_x+\kappa(v)\Big(-\frac{v_{xx}}{v^5}+\frac{5v_x^2}{2v^6}\Big)-\frac{\kappa'(v)v_x^2}{2v^5}\Big]_x\\
	&=\tau_2\Big(\mu(v)\frac{h_x}{v}\Big)_x +\Big[\Big(c\frac{\mu(v)}{v}\frac{\mu'(v)v-\mu(v)}{v^2}+\frac{5\kappa(v)}{2v^6}-\frac{\kappa'(v)}{2v^5}\Big)v_x^2+\Big(c\frac{\mu(v)^2}{v^2}-\frac{\kappa(v)}{v^5}\Big)v_{xx}\Big]\\
	&=\tau_2 \Big(\mu(v)\frac{h_x}{v}\Big)_x,
\end{align*}
where $\tau_2:=1-\tau_1$ and the second equality comes from the choice of the capillarity coefficient $\kappa(v) = c\mu(v)^2v^3$. Therefore, the effective velocity $h$ satisfies
\begin{equation}\label{eq:h_effect}
	h_t+p(v)_x = \tau_2 \Big(\mu(v)\frac{h_x}{v}\Big)_x.
\end{equation}
Combining \eqref{eq:v_effect} and \eqref{eq:h_effect}, the system \eqref{main0} can be transformed into the system of degenerate parabolic equations:
\begin{equation} \label{main00}
\begin{cases}
v_t - h_x = \t_1 \big(\m(v)\frac{v_x}{v}\big)_x = - \t_1(v^{\g-\a} p(v)_x)_x, \\
h_t + p(v)_x = \t_2 \big(\m(v)\frac{h_x}{v}\big)_x = \t_2\g (v^{-\a-1}h_x)_x.
\end{cases}
\end{equation}
Here, we also assumed \(b=\g\) for simplicity, as the value of \(b\) does not affect our analysis.

We now assume that the positive constants \(\t_1\) and \(\t_2\) satisfy 
\begin{equation} \label{0.9}
\t_1 \ge 9 \t_2, \quad\text{or equivalently,}\quad \tau_1 \ge 0.9.
\end{equation}
In terms of the constant $c$, this condition requires $c\le \frac{9}{100}$, as in the statements of the main theorems. 

On the other hand, the viscous-dispersive shocks in \((v,h)\)-variable corresponding to \eqref{shock_0} satisfy the following ODEs:
\begin{equation} \label{VS}
\begin{cases}
-\s_\e \rd_\x \vtil - \rd_\x \htil = -\t_1\rd_\x (\vtil^{\g-\a} \rd_\x p(\vtil)), \\
-\s_\e \rd_\x \htil + \rd_\x p(\vtil) = \t_2\g (\vtil^{-\a-1}\rd_\x \htil),\\
\lim_{\x\to \pm\infty} (\vtil,\htil)(\x) = (v_\pm,u_\pm).
\end{cases}
\end{equation}
For the sake of simplicity, we rewrite \eqref{main00} as the following system, based on the change of variables \((t,x)\mapsto (t,\x:=x-\s_\e t)\):
\begin{equation} \label{main}
\begin{cases}
v_t -\s_\e v_\x - h_\x = - \t_1(v^{\g-\a} p(v)_\x)_\x, \\
h_t -\s_\e h_\x + p(v)_\x = \t_2\g (v^{-\a-1}h_\x)_\x, \\
v|_{t=0}=v_0, \quad h|_{t=0}=h_0.
\end{cases}
\end{equation}
Recalling the space \(\Xcal_T\), the global solutions to \eqref{main} are in the function space
\begin{multline*}
\Hcal_T \coloneqq \{(v,h) \mid v-\underline{v}\in C([0,T];H^1(\RR)), h-\underline{u}\in C([0,T];H^1(\RR)),\\
0<v^{-1}\in L^\infty((0,T)\times\RR)\}
\end{multline*}
where \(\underline{v}\) and \(\underline{u}\) are in \eqref{sm-end}.

\vspace{2mm}
We now restate the Theorem \ref{thm_main} in an equivalent form by using $(v,h)$-variables.

\begin{theorem} \label{thm_main3}
Let \(\g,\a\) be any constants satisfying \eqref{g-a-condition} and \(\t_1,\t_2 \in (0,1)\) be any constants with \(\t_1+\t_2=1\) and \(\t_1 \ge 0.9\). For a given constant state \((v_-,u_-) \in \RR^+\times\RR\), there exist constants \(\e_0,\d_0>0\) such that the following holds.\\
Fix any \(\e\in(0,\e_0),\) \(\l\in(\d_0^{-1}\e,\d_0),\) and \((v_+,u_+)\in\RR^+\times\RR\) satisfying \eqref{end-con} with \(|p(v_+)-p(v_-)|=\e\). Let \(\Util\coloneqq(\vtil,\htil)\) be the viscous-dispersive shock defined in \eqref{VS} and let \(U\coloneqq(v,h)\) be a solution to \eqref{main} in \(\Hcal_T\) for some $T>0$ subject to an initial datum \(U_0\coloneqq(v_0,h_0)\) which satisfies \(\int_\RR \et(U_0|\Util)dx < \infty\). Then, there exist a smooth monotone function \(a\colon\RR\to\RR^+\) with \(\lim_{x\to\pm\infty}a(x)=1+a_\pm\) for some constants \(a_-\) and \(a_+\) with \(|a_+-a_-|=\l\), a shift function \(X\in W^{1,1}(0,T)\) with \(X(0)=0\) and a constant \(C>0\) (independent of \(\e,\l,\d_0\) and \(T\)) such that 
\begin{equation} \label{cont_main}
	\begin{aligned}
		&\int_\RR a(\x)\et(U(t,\x+X(t))|\Util(\x))d\x \\
		&\qquad
		+ \d_0\frac{\e}{\l}\int_0^T \int_\RR \abs{\s_\e a'(\x)} Q(v(s,\x+X(s))|\vtil(\x))d\x ds \\
		&\qquad
		+ \d_0 \t_1 \int_0^T \int_\RR a(\x) v^{\g-\a}(s,\x+X(s))\abs{\rd_\x \left(p(v(s,\x+X(s)))-p(\vtil(\x))\right)}^2 d\x ds \\
		&\qquad
		+ \d_0 \t_2\g \int_0^T \int_\RR a(\x)  v^{-\a-1}(s,\x+X(s)) |\rd_\x (h(s,\x+X(s))-\htil(\x))|^2 d\x ds\\
		&\le \int_\RR a(\x)\et(U_0(\x)|\Util(\x))d\x,
	\end{aligned}
\end{equation}
and
\begin{equation} \label{shift-control}
	\begin{aligned}
		&|\dot{X}(t)| \le \frac{C}{\e^2}(f(t)+1), \quad \text{for a.e. } t\in[0,T], \\
		&\text{for some positive function } f \text{ satisfying } \norm{f}_{L^1(0,T)} \le \frac{\l}{\d_0\e}\int_\RR \et(U_0(x)|\Util(x))dx.
	\end{aligned}
\end{equation}	
\end{theorem}

\begin{remark} \label{rmk:2shock}
It suffices to show Theorem \ref{thm_main3} for \(2\)-shocks, since the change of variables \(x\to -x\), \(u\to -u\), \(\s_\e\to -\s_\e\) implies the same result for \(1\)-shock.
Therefore, in what follows, we consider \(2\)-shocks \((\vtil,\htil)\), i.e., \(v_-<v_+,u_->u_+\) and \(\s_\e=\sqrt{-\frac{p(v_+)-p(v_-)}{v_+-v_-}}\).
\end{remark}

\(\bullet\) \textbf{Notation:} In what follows, \(C\) denotes a positive constant which may vary from line to line, but stays independent on \(\e\) (the strength of shock), \(\l\) (the total variation of the weight function \(a\)).

\section{Proof of Theorem \ref{thm_main3}}\label{sec:3}
\setcounter{equation}{0}
The proof of Theorem \ref{thm_main3} is based on the method of $a$-contraction with shifts (in short, $a$-contraction). 
The $a$-contraction method was first developed in \cite{KV-ARMA16,Vasseur-16} for the study of stability of Riemann shocks, and extended to the study for stability of large perturbations of viscous shock to viscous conservation laws as in \cite{EEK-BNSF,Kang-JMPA21,KO,KV-Poincare17,KV-JEMS21,KV-Inven21,KVW-CMP21} (see also \cite{BC-25,Kang-18} for the non-contraction property induced by the method).
We refer to \cite{CFK-2,CFK-1,CKV-ARMA22,GKV-JHDE23} for further applications of the a-contraction method to the stability analysis of inviscid shocks and small BV solutions to conservation laws.

In the present paper, we exploit this methodology to obtain the desired contraction property for the NSK system.
To this end, we first introduce several preliminaries and the method of $a$-contraction with shifts. Then, we present the main proposition (Proposition \ref{prop:main}) which is a key estimate to prove Theorem \ref{thm_main3}.

\subsection{Preliminaries for the method of $a$-contraction with shifts}

In this part, we collect technical estimates on the relative functionals and viscous-dispersive shock. Since the proof of these technical lemmas can be found in previous literature, we will omit the proofs and refer to the references for details.

\subsubsection{Global and local estimates on the relative quantities}
We first present useful inequalities on the relative quantities that will be crucially used in the later analysis. First of all, the following lemma provides global estimates on the relative functional \(Q(\cdot|\cdot)\) corresponding to the convex function \(Q(v)=\frac{v^{-\g+1}}{\g-1}\) for \(\g>1\) and \(Q(v)=-\log v\) for \(\g=1\) on \(\RR^+\).
\begin{lemma} [\cite{EEKO-ISO,KV-JEMS21}] \label{lem-pro}
For given constants \(\g\ge1\) and \(v_->0\), there exist constants \(c_1,c_2>0\) such that the following inequalities hold.
\begin{enumerate}
\item For any \(w\in(0,2v_-)\),
\begin{equation} \label{rel_Phi}
\begin{aligned}
Q(v|w) &\ge c_1\abs{v-w}^2 \quad 
&&\text{for all } 0<v\le 3v_-, \\
Q(v|w) &\ge c_2 \abs{v-w} \quad 
&&\text{for all } v\ge 3v_-.
\end{aligned}
\end{equation}
\item If \(0<w\le u \le v\) or \(0<v \le u \le w\), then
\begin{equation} \label{Phi-sim}
Q(v|w) \ge Q(u|w),
\end{equation} 
and for any \(\d_* \in (0,w/2)\), there exists a constant \(c_3>0\) such that if, in addition, \(w\in[\frac{1}{2}v_-,2v_-]\), \(\abs{w-v} > \d_*\), and \(\abs{w-u}\le\d_*\), then
\begin{equation} \label{rel_Phi-L}
Q(v|w)-Q(u|w) \ge c_3 \abs{v-u}.
\end{equation}
\end{enumerate}
\end{lemma}
\begin{proof}
We refer to \cite[Lemma 2.4]{KV-JEMS21} for the case \(\g>1\) and to \cite[Lemma 3.1]{EEKO-ISO} for \(\g=1\).
\end{proof}

The following lemma provides local estimates on the relative functionals.
\begin{lemma}[\cite{EEKO-ISO,KV-JEMS21}] \label{lem:local}
For given constants \(\g\ge1\) and \(v_->0\), there exist constants \(C>0\) and \(\d_*>0\) such that
for any \(\d \in (0,\d_*) \), the following is true.
\begin{enumerate}
\item For any \(v,w\in\RR^+\) satisfying \(\abs{p(v)-p(w)}<\d\) and \(\abs{p(w)-p(v_-)}<\d\),
\begin{equation} \label{Q-est-U}
Q(v|w) \le \bigg(\frac{p(w)^{-1/\g-1}}{2\g}+C\d\bigg) \abs{p(v)-p(w)}^2,
\end{equation}
\begin{equation} \label{Q-est-L}
Q(v|w) \ge \frac{p(w)^{-1/\g-1}}{2\g} \abs{p(v)-p(w)}^2
- \frac{1+\g}{3\g^2} p(w)^{-1/\g-2} (p(v)-p(w))^3,
\end{equation}
\begin{equation} \label{p-est1}
p(v|w) \le \left(\frac{\g+1}{2\g}\frac{1}{p(w)}+C\d\right) \abs{p(v)-p(w)}^2,
\end{equation}
where the relative pressure is defined as 
\[
p(v|w) =p(v)-p(w)-p'(w)(v-w).
\]
\item For any \(v,w\in\RR^+\) with \(\abs{p(w)-p(v_-)} \le \d\), if they satisfy either \(Q(v|w)<\d\) or \(\abs{p(v)-p(w)} < \d\), then
\begin{equation} \label{p-quad}
\abs{p(v)-p(w)}^2 \le C Q(v|w).
\end{equation}
\end{enumerate}
\end{lemma}
\begin{proof}
For the proof, we refer to \cite[Lemma 2.6]{KV-JEMS21} and \cite[Lemma 3.2]{EEKO-ISO}.
\end{proof}

\begin{remark} \label{rmk:rel}
	The locally quadratic structure of the relative functional guarantees that for any compact set \(K\subset\RR^+\), there exists a constant \(C>0\) such that
	\begin{equation*}
		C^{-1}\abs{v-w}^2 \le Q(v|w) \le C\abs{v-w}^2 \quad 
		\text{ for any } v,w \in K.
	\end{equation*}
	Note that the constant \(C\) depends on the compact set \(K\).
\end{remark}

\subsubsection{Properties of small shock waves}
We now introduce properties of viscous-dispersive shocks, under the assumption that the shock amplitude is sufficiently small.
In the sequel, due to translation invariance property, we assume \(\vtil(0)=\frac{v_-+v_+}{2}\) without loss of generality.

\begin{lemma} \label{lem-VS}
For a given constant state \(U_-\coloneqq(v_-,u_-)\in\RR^+\times\RR\), there exist constants \(\e_0,C,C_1,C_2>0\) such that the following holds. 

\begin{enumerate}
\item There exists a monotonic viscous-dispersive shock profile \(\Util\coloneqq(\vtil,\htil)\) satisfying \eqref{VS} with amplitude \(\e \coloneqq |p(v_-)-p(v_+)|<\e_0\) and \(\vtil(0)=\frac{v_-+v_+}{2}\).
\item It holds that
\begin{equation} \label{tail}
C^{-1} \e^2 e^{-C_1\e \abs{\x}} \le \vtil'(\x) \le C \e^2 e^{-C_2 \e \abs{\x}}, \quad \forall \x\in\RR.
\end{equation}
As a consequence,
\begin{equation} \label{lower-v}
\inf_{\x\in[-\frac{1}{\e},\frac{1}{\e}]} \vtil'(\x) \ge C\e^2.
\end{equation}
\item It also holds that \(|\vtil'(\x)| \sim |\htil'(\x)|\) for all \(\x\in\RR\).
In particular,
\begin{equation} \label{ratio-vh}
\begin{aligned}
|\s_* \vtil'(\x) + \htil'(\x)|+|\s_* \htil'(\x) - p(\vtil(\x))'| \le C \e \abs{\vtil'(\x)}, \quad \forall \x\in\RR,
\end{aligned}
\end{equation}
where \(p(\vtil)'=\frac{d}{d\x}p(\vtil)\) and
\begin{equation} \label{s*}
\s_* \coloneqq \sqrt{-p'(v_-)}
\end{equation}
which satisfies
\begin{equation} \label{sm1}
\abs{\s_\e - \s_*} \le C \e.
\end{equation}
\item Finally, 
\begin{equation} \label{2nd-der}
|\vtil''(\x)|+|\htil''(\x)| \le C \e |\vtil'(\x)|, \quad \forall \x\in\RR.
\end{equation}
\end{enumerate} 
\end{lemma}
\begin{proof}
The proof can be found in \cite{HKKL-JDE25} for the case where the coefficients \(\m\) and \(\k\) are positive constants.
We also refer to \cite{EE-BNSF,EEKO-BNSF} for a rigorous argument.
As the same argument carries over the non-constant coefficient case, we omit the details here and refer to \cite{FLP-NSK}.
\end{proof}

To address arbitrarily large perturbations, one may exploit the exponential decay of the derivatives of viscous-dispersive shocks.
The lemmas below furnish the required tools.

\begin{lemma}[{\cite[Lemma 4.2]{EEK-BNSF}}]\label{lemma_pushing}
Let \(f\colon \RR\to \RR\) be a smooth function with 
\(
\int_\RR \vtil'(\x) \abs{f(\x)} d\x <\infty.
\) Then, under the same assumption in Lemma \ref{lem-VS}, there is a constant \(C>0\) such that for any \(\x_0\in[-1/\e, 1/\e]\), 
\[
\abs{\int_\RR \vtil'(\x) \int_{\x_0}^{\x} f(\z) d\z d\x} 
\le \frac{C}{\e}\int_\RR \vtil'(\z) \abs{f(\z)} d\z.
\]
\end{lemma}

\begin{lemma}[{\cite[Lemma 4.3]{EEK-BNSF}}]\label{lemma_Linfty}
Let \(f\colon \RR\to \RR\) be a smooth function with 
\(
\int_\RR \vtil'(\x) \abs{f(\x)} d\x <\infty.
\) Then, under the same assumption in Lemma \ref{lem-VS}, there is a constant \(C>0\) such that for any \(\x_0\in[-1/\e, 1/\e]\),
\[
\abs{\vt'(\x)\int_{\x_0}^{\x} f(\z)d\z} \le C\int_\RR \vt'(\z)\abs{f(\z)}d\z.
\]
\end{lemma}

\subsection{Relative entropy method}
Our approach relies on the relative entropy method, a nonlinear energy method which allows us to handle large perturbations.
This method was originally introduced by Dafermos \cite{Dafermos-ARMA79} and DiPerna \cite{Diperna-Indiana79} to prove \(L^2\)-stability and uniqueness of Lipschitz solutions to the hyperbolic conservation laws endowed with a convex entropy.

To apply the relative entropy method, we first rewrite \eqref{main} into the form of general viscous hyperbolic system of conservation laws:
\begin{equation} \label{VCL}
\rd_t U + \rd_x A(U) =
\begin{pmatrix}
- \t_1 \rd_x (v^\b \rd_x(p(v))) \\
\t_2\g \rd_x (v^{-\a-1} \rd_x h)
\end{pmatrix},
\end{equation}
where \(\b\coloneqq\g-\a\) and
\[
U\coloneqq
\begin{pmatrix}
v \\ h
\end{pmatrix}, \quad
A(U)\coloneqq
\begin{pmatrix}
- h \\ p(v)
\end{pmatrix}.
\]
It is known that the system \eqref{VCL} admits a convex entropy:
\begin{align*}
\et(U) \coloneqq
\begin{cases}
\frac{h^2}{2}-\log v, \quad \text{ for } \g=1,  \\  
\frac{h^2}{2}+Q(v), \quad \text{ for } \g>1,
\end{cases}
\end{align*}
where \(Q(v)=\frac{v^{-\g+1}}{\g-1}\), i.e., \(Q'(v)=-p(v)\).
The relative entropy for \eqref{VCL} is given by
\[
\et(U|\Util) = \et(U) - \et(\Util) - D\et(\Util)(U-\Util) = \frac{1}{2}(h-\htil)^2 + Q(v|\vtil).
\]
Throughout the proof, we consider a weighted relative entropy, which measures a solution to \eqref{VCL} from the shock \(\Util\) up to a shift \(X\) (which will be determined later):
\[
a(\x) \et(U(t,\x+X(t))|\Util(\x)).
\]
For simplicity, we introduce the following notations: for any function \(f\colon \RR^+\times\RR\to\RR\) and the shift function \(X(t)\),
\[
f^{\pm X}(t,\x) \coloneqq f(t,\x\pm X(t)).
\]
In the following lemma, we extract a quadratic structure from the evolution of the weighted relative entropy.

\begin{lemma} \label{lem-rel}
Let \(a\) be any positive smooth bounded function whose derivative is bounded and integrable.
Let \(U\in\Hcal_T\) be a solution to \eqref{VCL}, and \(X\) be any absolutely continuous function on \([0,T]\).
Then, we have
\begin{equation} \label{ineq-0}
\begin{aligned}
&\frac{d}{dt} \int_\RR a(\x) \et(U^X(t,\x)|\Util(\x))d\x \\
&\qquad
=\dot{X}(t)Y(U^X) + \Jcal^{bad}(U^X) + \Pcal_1(U^X) + \Pcal_2(U^X) - \Jcal^{good}(U^X),
\end{aligned}
\end{equation}
where
\begin{equation}
\begin{aligned}\label{ybg-first}
Y(U) &\coloneqq -\int_\RR a'\et(U|\Util)d\x + \int_\RR a \rd_\x \nabla \et(\Util)(U-\Util) d\x, \\
\Jcal^{bad}(U) &\coloneqq \int_\RR a' (p(v)-p(\vtil))(h-\htil) d\x
+ \s_\e \int_\RR a \vtil' p(v|\vtil) d\x, \\
\Pcal_1(U) &\coloneqq 
- \t_1\int_\RR a' v^\b (p(v)-p(\vtil)) \rd_\x(p(v)-p(\vtil)) d\x \\
&\qquad
- \t_1\int_\RR a'(p(v)-p(\vtil)) (v^\b-\vtil^\b) \rd_\x p(\vtil) d\x \\
&\qquad
- \t_1\int_\RR a \rd_\x(p(v)-p(\vtil)) (v^\b-\vtil^\b) \rd_\x p(\vtil) d\x, \\
\Pcal_2(U) &\coloneqq
-\t_2\g\int_\RR a' \frac{1}{v^{\a+1}} (h-\htil) \rd_\x(h-\htil) d\x \\
&\qquad
+\t_2\g\int_\RR a(h-\htil)\Big(\frac{1}{v^{\a+1}}-\frac{1}{\vtil^{\a+1}}\Big) \htil'' d\x \\
&\qquad
+\t_2\g\int_\RR a (h-\htil) \rd_\x\Big(\frac{1}{v^{\a+1}}-\frac{1}{\vtil^{\a+1}}\Big) \htil' d\x,\\
\Jcal^{good}(U) &\coloneqq \s_\e\int_\RR a' Q(v|\vtil) d\x
+ \frac{\s_\e}{2}\int_\RR a' (h-\htil)^2 d\x\\
&\qquad
+ \t_1\int_\RR a v^\b |\rd_\x (p(v)-p(\vtil))|^2 d\x
+ \t_2\g\int_\RR a \frac{1}{v^{\a+1}}|\rd_\x (h-\htil)|^2 d\x.
\end{aligned}
\end{equation}
\end{lemma}

\begin{proof}
The system \eqref{VCL} and the barotropic NS system in Lagrangian mass coordinates (with effective velocity) share similar structure, except for additional dissipation in the velocity equation.
We therefore concentrate on the part related to velocity dissipation, and omit the remaining details and refer to \cite[Lemma 4.2]{KV-Inven21} and \cite[Lemma 3.5]{EEKO-ISO}. The velocity dissipation term yields the following terms
\begin{align*}
&\t_2\g\int_\RR a (h-\htil)\Big(\frac{1}{v^{\a+1}}h_\x-\frac{1}{\vtil^{\a+1}}\htil_\x\Big)_\x d\x
=-\t_2\g\int_\RR \big(a (h-\htil)\big)_\x \Big(\frac{1}{v^{\a+1}}h_\x-\frac{1}{\vtil^{\a+1}}\htil_\x\Big) d\x\\
&=-\t_2\g\int_\RR \big(a (h-\htil)\big)_\x \frac{1}{v^{\a+1}}(h-\htil)_\x d\x
+\t_2\g\int_\RR a (h-\htil) \Big(\Big(\frac{1}{v^{\a+1}}-\frac{1}{\vtil^{\a+1}}\Big) \htil'\Big)_\x d\x,
\end{align*}
which gives \(\Pcal_2\) and the third term of \(\Jcal^{good}\).
\end{proof}
\begin{remark} 
	The weight function \(a\) will be defined so that \(\s_\e a'>0\).
	Hence, \(-\Jcal^{good}\) consists of the terms with negative signs, while \(\Jcal^{bad},\Pcal_1\) and \(\Pcal_2\) consist of the terms with indefinite signs.
\end{remark}

\subsection{Construction of the weight function and shift}
We now define the weight function \(a\) as
\begin{equation} \label{a-def}
a(\x) \coloneqq 1 - \frac{\l}{\e} (p(\vtil(\x))-p(v_-)),
\end{equation}
whose total variation is \(\l>0\).
Thanks to the properties of shocks, the weight \(a\) also possesses useful properties. More precisely, \(a\) increases from \(1\) to \(1+\l\) on \(\RR\), and
\begin{equation} \label{totvar-ai}
\int_\RR |a'(\x)| d\x = \l.
\end{equation}
Moreover, since \(|p'(v_+)| \le |p'(\vtil)| \le |p'(v_-)|\) and
\begin{equation} \label{der-a}
a'(\x) = - \frac{\l}{\e}p'(\vtil(\x)) \vtil'(\x)>0,
\end{equation}
it follows from \eqref{tail} and \eqref{ratio-vh} that
\begin{equation} \label{der-scale}
|a'|
\sim \frac{\l}{\e}|\vtil'|
\sim \frac{\l}{\e}|\htil'|, \qquad
|a'(\x)| \le C \e \l e^{-C\e \abs{\x}} \quad \text{for all } \x\in\RR.
\end{equation}

Next, we define the shift function \(X\) so that the right-hand side of \eqref{ineq-0} is negative. For each \(\e>0\), we first consider a continuous function \(\Phi_\e\) given by
\begin{equation} \label{Phi}
	\Phi_\e(y)=
	\begin{cases}
		\frac{1}{\e^2}, &\text{if } y\le-\e^2,\\
		-\frac{1}{\e^4}y, &\text{if } \abs{y}\le\e^2,\\
		-\frac{1}{\e^2}, &\text{if } y\ge\e^2.
	\end{cases}
\end{equation}
Then, we define a shift \(X\) as a solution to the following nonlinear ODE:
\begin{equation} \label{def-shift}
	\begin{cases}
		\dot{X}(t) = \Phi_\e(Y(U))\big(2|\Jcal^{bad}(U)|+2\abs{\Pcal_1(U)}+2\abs{\Pcal_2(U)}+1\big), \\
		X(0)=0,
	\end{cases}
\end{equation}
where \(Y,\Jcal^{bad},\Pcal_1\) and \(\Pcal_2\) are as in \eqref{ybg-first}.
Notice that for any solution \(U\in\Hcal_T\), this ODE admits a unique absolutely continuous solution on \([0,T]\).
We refer to \cite[Lemma A.1]{CKKV-M320} and \cite[Appendix C]{KV-JDE22} for detailed accounts of the existence and uniqueness theory of the ODE.

Then, this choice of the shift $X$, together with \eqref{ineq-0}, yields that
\begin{align}
\begin{aligned}\label{rel_ent_est}
	&\frac{d}{dt} \int_\RR a(\x) \et(U^X(t,\x)|\Util(\x))d\x \\
	&\qquad
	=\dot{X}(t)Y(U^X) + \Jcal^{bad}(U^X) + \Pcal_1(U^X) + \Pcal_2(U^X) - \Jcal^{good}(U^X)\\
	&\qquad=\Phi_\e(Y(U))\big(2|\Jcal^{bad}(U)|+2\abs{\Pcal_1(U)}+2\abs{\Pcal_2(U)}+1\big)Y(U^X) \\
	&\quad\qquad+ \Jcal^{bad}(U^X) + \Pcal_1(U^X) + \Pcal_2(U^X) - \Jcal^{good}(U^X)=:\Fcal(U).
\end{aligned}
\end{align}
Our goal is to show that $\Fcal(U)\le 0$, which yields the contraction of the weighted relative entropy. Thanks to \eqref{Phi}, we have
\begin{equation} \label{XdotY}
	\begin{aligned}
		&\Phi_\e(Y(U))\big(2|\Jcal^{bad}(U)|+2\abs{\Pcal_1(U)}+2\abs{\Pcal_2(U)}+1\big)Y(U)\\
		&\qquad \le
		\begin{cases}
			-2|\Jcal^{bad}(U)|-2\abs{\Pcal_1(U)}-2\abs{\Pcal_2(U)}, &\text{if } \abs{Y(U)} \ge \e^2,\\
			-\frac{1}{\e^4}\abs{Y(U)}^2, &\text{if } \abs{Y(U)} \le \e^2.
		\end{cases}
	\end{aligned}
\end{equation}
Then, for any \(U\in\Hcal_T\) satisfying \(|Y(U)|\ge\e^2\), we have 
\begin{equation}\label{FU-0}
	\Fcal(U)
	\le -|\Jcal^{bad}(U)|-\abs{\Pcal_1(U)}-\abs{\Pcal_2(U)} - \Jcal^{good}(U) \le 0.
\end{equation}
On the other hand, when \(|Y(U)|\le\e^2\), we use \eqref{Phi} to obtain
\begin{equation}\label{FU}
	\Fcal(U)\le -\frac{1}{\e^4}\abs{Y(U)}^2 + \Jcal^{bad}(U^X) + \Pcal_1(U^X) + \Pcal_2(U^X) - \Jcal^{good}(U^X).
\end{equation}
Thus, it suffice to show that the right-hand side of \eqref{FU} is negative, whenever $|Y(U)|\le \e^2$.

\subsection{Maximization in terms of \(h-\htil\)}
In order to show that the right-hand side of \eqref{FU} is negative for $|Y(U)|\le \e^2$, we need a sharp estimate for \(p(v)-p(\vtil)\) near \(p(\vtil)\).
To this end, it is necessary to express \(\Jcal^{bad}\) on the right-hand side of \eqref{FU} in terms of \(p(v)\) near \(p(\vtil)\), by separating \(h-\htil\) from the first term of \(\Jcal^{bad}\).
Therefore, we will rewrite \(\Jcal^{bad}\) into maximized representation with respect to \(h-\htil\) as in the following lemma.
Nonetheless, we will retain all terms of \(\Jcal^{bad}\) within the region \(\{p(v)-p(\vtil)>\d\}\) for small values of \(v\), since we use the estimate \eqref{bo1m} to control the first term of \(\Jcal^{bad}\) in that region.
The values of \(\d\) will be determined later.

\begin{lemma} \label{lem-max}
Let \(a\colon\RR\to\RR^+\) be the weight function in \eqref{a-def}
and \(\Util=(\vtil,\htil)\) be the viscous-dispersive shock in \eqref{VS}.
Let \(\d>0\) be any constant.
Then, for any \(U\in\Hcal_T\),
\begin{equation} \label{ineq-1}
\Jcal^{bad}(U) - \Jcal^{good}(U) = \Bcal_{\d}(U) - \Gcal_{\d}(U),
\end{equation}
where\begin{equation}
\begin{aligned}\label{badgood}
\Bcal_\d(U) &\coloneqq \frac{1}{2\s_\e}\int_\RR a'(p(v)-p(\vtil))^2\one{p(v)\le p(\vtil)+\d} d\x
+\s_\e \int_\RR a \vtil' p(v|\vtil) d\x \\
&\qquad
+\int_\RR a'(p(v)-p(\vtil))(h-\htil)\one{p(v)>p(\vtil)+\d} d\x \\
\Gcal_\d(U) &\coloneqq \frac{\s_\e}{2}\int_\RR a'\Big(h-\htil-\frac{p(v)-p(\vtil)}{\s_\e}\Big)^2 \one{p(v)\le p(\vtil)+\d} d\x \\
&\qquad
+\frac{\s_\e}{2}\int_\RR a'(h-\htil)^2 \one{p(v)>p(\vtil)+\d} d\x
+\s_\e\int_\RR a' Q(v|\vtil) d\x\\
&\qquad+\t_1\int_\RR a v^\b |\rd_\x (p(v)-p(\vtil))|^2 d\x
+\t_2\g\int_\RR a \frac{1}{v^{\a+1}} |\rd_\x (h-\htil)|^2 d\x.
\end{aligned}
\end{equation}
\end{lemma}

\begin{proof}
The proof is essentially the same as \cite[Lemma 4.3]{KV-Inven21}.
Thus, we omit it.
\end{proof}
\begin{remark} \label{rmk:good}
	Since \(\s_\e a' > 0\) and \(a>0\), \(-\Gcal_\d\) still consists of good terms.
\end{remark}

\subsection{Main proposition}
We now state the main proposition for the proof of Theorem \ref{thm_main3} and completes the proof of Theorem \ref{thm_main} as an immediate consequence. The following proposition is the key estimate for the proof of Theorem \ref{thm_main3}. 
\begin{prop} \label{prop:main}
For a given constant state \(U_-=(v_-,u_-)\in\RR^+\times\RR\), there exist constants \(\e_0,\d_0,\d_3\in(0,1/2)\) such that for any \(\e\in(0,\e_0)\) and any \(\l\in(\d_0^{-1}\e,\d_0)\), the following holds.\\
For all \(U\in\Hcal_T\) which satisfies \(\abs{Y(U)}\le\e^2\),
Then, it holds that
\begin{equation} \label{contraction}
\begin{aligned}
&\Rcal(U) \coloneqq 
-\frac{1}{\e^4}\abs{Y(U)}^2 + \Bcal_{\d_3}(U) + \Pcal_1(U) + \Pcal_2(U)\\
&\qquad\qquad
-\Gcal_h^-(U) -\Gcal_h^+(U) -\Big(1-\d_0\frac{\e}{\l}\Big)\Gcal_p(U)
-\frac{1}{2}\Dcal_h(U) -(1-\d_0)\Dcal_p(U) \le 0,
\end{aligned}
\end{equation}
where \(Y,\Bcal_{\d_3},\Pcal_1\) and \(\Pcal_2\) are as in \eqref{ybg-first} and \eqref{badgood}, \(\Gcal_h^-,\Gcal_h^+,\Gcal_p,\Dcal_h\) and \(\Dcal_p\) denote the good terms of \(\Gcal_{\d_3}\) in \eqref{badgood} as follows:
\begin{equation} \label{def-GD}
\begin{aligned}
\Gcal_h^-(U) &\coloneqq
\frac{\s_\e}{2}\int_\RR a'(h-\htil)^2 \one{p(v)>p(\vtil)+\d_3} d\x, \\
\Gcal_h^+(U) &\coloneqq
\frac{\s_\e}{2}\int_\RR a'\Big(h-\htil-\frac{p(v)-p(\vtil)}{\s_\e}\Big)^2 \one{p(v)\le p(\vtil)+\d_3} d\x, \\
\Gcal_p(U) &\coloneqq
\s_\e\int_\RR a' Q(v|\vtil) d\x, \\
\Dcal_h(U) &\coloneqq
\t_2\g\int_\RR a \frac{1}{v^{\a+1}} |\rd_\x (h-\htil)|^2 d\x, \\
\Dcal_p(U) &\coloneqq
\t_1\int_\RR a v^\b |\rd_\x (p(v)-p(\vtil))|^2 d\x.
\end{aligned}
\end{equation}
Moreover, the following also holds:
\begin{equation} \label{bad-bound2}
|\Jcal^{bad}(U)|
+\abs{\Pcal_1(U)}
+\abs{\Pcal_2(U)}
\le \Dcal_h(U)+\Dcal_p(U)+C\d_0.
\end{equation}
\end{prop}

The proof of Proposition \ref{prop:main} will be presented in the next section. We now show that Proposition \ref{prop:main} yields Theorem \ref{thm_main3}, which in turn implies Theorem \ref{thm_main} as an immediate consequence. To prove Theorem \ref{thm_main3}, we only need to show that the right-hand side of \eqref{FU} is negative under the condition $|Y(U)|\le \e^2$. However, since $\Gcal_{\delta_3}=\Gcal_h^-+\Gcal_h^++\Gcal_p+\Dcal_h+\Dcal_p$, we use Proposition \ref{prop:main} with \eqref{ineq-1} to obtain
\begin{align}
\begin{aligned}\label{FU-1}
\Fcal(U)
&\le -\frac{1}{\e^4}\abs{Y(U)}^2 + \Bcal_{\d_3}(U)
+\Pcal_1(U)+\Pcal_2(U)
- \Gcal_{\d_3}(U)\\
&\le -\d_0\frac{\e}{\l}\Gcal_p(U) - \frac{1}{2}\Dcal_h(U) - \d_0\Dcal_p(U) \le 0. 
\end{aligned}
\end{align}
Thus, combining \eqref{FU-0} for $|Y(U)|\ge \e^2$, \eqref{FU-1} for $|Y(U)|\le \e^2$ and \eqref{ybg-first}, we find that for all \(U\in\Hcal_T\),
\begin{align*}
\Fcal(U) + \d_0\frac{\e}{\l}\Gcal_p(U) + \d_0(\Dcal_h(U)+\Dcal_p(U))
+(|\Jcal^{bad}(U)|+\abs{\Pcal_1(U)}+\abs{\Pcal_2(U)})\oone{\{Y(U)\ge \e^2\}} \le 0.
\end{align*}
Therefore, from \eqref{rel_ent_est}, we find that for a.e. \(t>0\),
\begin{multline} \label{RB-Gronwall}
\frac{d}{dt}\int_\RR a\et(U^X|\Util) d\x
+\d_0\frac{\e}{\l}\Gcal_p(U) + \d_0(\Dcal_h(U)+\Dcal_p(U))\\
+(|\Jcal^{bad}(U)|+\abs{\Pcal_1(U)}+\abs{\Pcal_2(U)})\oone{\{Y(U)\ge \e^2\}} \le 0.
\end{multline}
Integrating both sides in time yields
\begin{multline} \label{WhoisL1}
\int_\RR a\et(U^X|\Util) d\x
+\d_0\frac{\e}{\l} \int_0^T \Gcal_p(U)dt + \d_0 \int_0^T(\Dcal_h(U)+\Dcal_p(U))dt\\
+\int_0^T (|\Jcal^{bad}(U)|+\abs{\Pcal_1(U)}+\abs{\Pcal_2(U)})\oone{\{Y(U)\ge \e^2\}}dt
\le\int_\RR a\et(U_0|\Util) d\x < \infty,
\end{multline}
which establishes \eqref{cont_main}.

\vspace{2mm}
It now remains to prove \eqref{shift-control}.
To this end, we first note by \eqref{Phi} and \eqref{def-shift} that
\begin{equation} \label{X-upper}
|\dot{X}| \le \frac{1}{\e^2} (2|\Jcal^{bad}(U)|+2\abs{\Pcal_1(U)}+2\abs{\Pcal_2(U)}+1),
\quad \text{ for a.e. } t\in(0,T).
\end{equation}
Thanks to \eqref{WhoisL1}, \(\Dcal_h(U)\) and \(\Dcal_p(U)\) are in \(L^1(0,T)\), and thus \eqref{bad-bound2} implies
\[
(2|\Jcal^{bad}(U)|+2\abs{\Pcal_1(U)}+2\abs{\Pcal_2(U)})\oone{\{Y(U)\le \e^2\}}  \in L^1(0,T).
\]
This together with \eqref{WhoisL1}--\eqref{X-upper} yields \eqref{shift-control}, which completes the proof of Theorem \ref{thm_main3}. \qed\\

Therefore, it remains to show that Proposition \ref{prop:main} holds, which will be treated in the following section.

\section{Proof of Proposition \ref{prop:main}}\label{sec:4}
\setcounter{equation}{0}
In this section, we present the proof of Proposition \ref{prop:main}. 
The proof consists of two parts. First we apply a nonlinear Poincar\'e inequality to derive precise estimates when \(p(v)\) is close to \(p(\vtil)\) in Section \ref{sec:in}. We exploit the smallness of the localized relative entropy together with interpolation-type pointwise estimates when $p(v)$ is not close to $p(\vtil)$ in Section \ref{sec:out}.

\subsection{Sharp estimate near a shock wave} \label{sec:in}
In this subsection, we provide a sharp estimate when \(p(v)\) exhibits close values to \(p(\vtil)\).
To this end, we define the following functionals:
\begin{equation} \label{note-in}
\begin{split}
&\begin{aligned} 
&\Ical_g^Y(v) \coloneqq
-\frac{1}{2\s_\e^2} \int_\RR a'\abs{p(v)-p(\vtil)}^2d\x - \int_\RR a' Q(v|\vtil) d\x \\
&\qquad\qquad\qquad
-\int_\RR a p(\vtil)' (v-\vtil) d\x + \frac{1}{\s_\e} \int_\RR a \htil' (p(v)-p(\vtil)) d\x
\end{aligned} \\
&\begin{aligned}
&\Ical_1(v) \coloneqq 
\frac{1}{2\s_\e} \int_\RR a' (p(v)-p(\vtil))^2 d\x,
&&\Ical_2(v) \coloneqq
\s_\e \int_\RR a \vtil' p(v|\vtil) d\x, \\
&\Gcal_p(v) \coloneqq
\s_\e \int_\RR a' Q(v|\vtil) d\x,
&&\Dcal_p(v) \coloneqq
\t_1\int_\RR a v^\b \abs{\rd_\x (p(v)-p(\vtil))}^2 d\x.
\end{aligned}
\end{split}
\end{equation}
Note that all of these quantities depend only on \(v\), not on \(h\). The following proposition provides a sharp estimate when \(p(v)\) is close to \(p(\vtil)\), which will become one of the key parts to prove Proposition \ref{prop:main} (See Section \ref{sec:proof}).

\begin{prop} \label{prop:main3}
For any constant \(C_2>0\), there exist constants \(\e_0,\d_3>0\) such that for any \(\e \in (0,\e_0)\) and any \(\l,\d \in (0,\d_3]\) with \(\e \le \l\), the following statement holds. \\
For any function \(v \colon \RR \to \RR^+\) such that \(\Dcal_p(v)+\Gcal_p(v)\) is finite, if
\begin{equation} \label{assYp}
\abs{\Ical_g^Y(v)} \le C_2 \frac{\e^2}{\l},\qquad  \norm{p(v)-p(\vtil)}_{L^\infty(\RR)}\le \d_3,
\end{equation}
then
\begin{align}
\begin{aligned}\label{redelta}
\Rcal_{\e,\d}(v)
&\coloneqq -\frac{1}{\e\d}\abs{\Ical_g^Y(v)}^2
+\Ical_1(v)
+\Ical_2(v)-\left(1-\d \frac{\e}{\l} \right)\Gcal_p(v)-(1-\d)\Dcal_p(v)\le 0.
\end{aligned}
\end{align}
\end{prop}

The proof of Proposition \ref{prop:main3} requires the following nonlinear Poincar\'e-type inequality.
\begin{prop}[{\cite[Proposition 5.1]{EE-BNSF}}] \label{prop_NPI}
Let \(C_1>0\) be any constant.
Then, there exists a constant \(\d_2>0\) such that for any \(\d\in(0,\d_2)\), the following holds. \\
For any function \(W\in L^2(0,1)\) such that \(\sqrt{y(1-y)}\rd_y W \in L^2(0,1)\), if \(\int_0^1 \abs{W(y)}^2 dy \le C_1\), then the following inequality holds:
\begin{align*}
\Rcal_\d(W)\coloneqq
&-\frac{1}{\d} \left(\int_0^1 W^2 dy + 2\int_0^1 W dy \right)^2
+(1+\d)\int_0^1 W^2 dy \\
&\qquad
+\frac{2}{3}\int_0^1 W^3 dy
+\d\int_0^1 \abs{W}^3 dy
-(0.9-\d) \int_0^1 y(1-y)\abs{\rd_y W}^2 dy
\le 0.
\end{align*}
\end{prop}

To use Proposition \ref{prop_NPI} in the proof of Proposition \ref{prop:main3}, we need to rewrite all the functionals in \eqref{note-in} in terms of a bounded variable \(y\in[0,1]\).
To this end, we will consider the following variables:
\begin{equation} \label{yandw}
y=\frac{p(v_-)-p(\vtil(\x))}{\e}, \qquad
w \coloneqq (p(v)-p(\vtil)) \circ y^{-1}, \qquad
W \coloneqq \frac{\l}{\e} w.
\end{equation}
The following lemma provides a Jacobian estimate corresponding to the variables in \eqref{yandw}.

\begin{lemma} \label{lem-Jacobian}
Let \(\a_\g \coloneqq \frac{\g\s_*p(v_-)}{\g+1}\).
Then, there exists a constant \(\e_0>0\) such that for any \(\e \coloneqq \abs{p(v_-)-p(v_+)} <\e_0\), the following holds:
\begin{equation} \label{jacob}
\abs{\frac{\vtil^\b}{y(1-y)}\frac{dy}{d\x} - \frac{\e}{2\a_\g}} \le C \e^2.
\end{equation}
\end{lemma}

\begin{proof}
First of all, we integrate the shock equations \eqref{VS} over \((\pm\infty,\x]\) to find that 
\[
\begin{cases}
-\s_\e (\vtil-v_\pm) -(\htil-h_\pm) = -\t_1 \vtil^\b \rd_\x(p(\vtil)),\\
-\s_\e (\htil-h_\pm) +(p(\vtil)-p(v_\pm)) = \t_2\g \vtil^{-\a-1} \rd_\x \htil.
\end{cases}
\]
Using the above two equations, we have 
\[
\t_1 \vtil^\b \rd_\x p(\vtil) + \frac{\t_2\g}{\s_\e} \vtil^{-\a-1} \rd_\x \htil
= \frac{1}{\s_\e} \left( \s_\e^2 (\vtil - v_\pm) + (p(\vtil) - p(v_\pm))\right).
\]
Thanks to \eqref{ratio-vh} and \eqref{s*}--\eqref{sm1}, the left-hand side can be estimated as follows: 
\[
\abs{\t_1 \vtil^\b p(\vtil)' + \frac{\t_2\g}{\s_\e} \vtil^{-\a-1} \htil'
-(\t_1+\t_2)\vtil^\b p(\vtil)'} \le C\e|\vtil'|,
\]
where \(p(\vtil)'=\frac{d}{d\x}p(\vtil)\). Therefore, using $\tau_1+\tau_2=1$, we have
\[\left|\vtil^\beta p(\vtil)'-\frac{1}{\s_\e}(\s_\e^2(\vtil-v_\pm)+(p(\vtil)-p(v_\pm)))\right|\le C\e|\vtil'|.\]
On the other hand, since \eqref{yandw} implies
\[
\frac{\vtil^\b}{y(1-y)}\frac{dy}{d\x} = \vtil^\b \Big(\frac{1}{y}\frac{dy}{d\x} + \frac{1}{1-y}\frac{dy}{d\x}\Big)
= \frac{\vtil^\b p(\vtil)'}{p(\vtil) - p(v_-)} - \frac{\vtil^\b p(\vtil)'}{p(\vtil) - p(v_+)},
\]
we take all these estimates above into account and make use of Taylor expansion with \eqref{end-con} and \eqref{sm1} to choose a constant \(\e_0>0\) (smaller than the one in Lemma \ref{lem-VS}) such that for any \(\e=p(v_-)-p(v_+)<\e_0\),
\[
\abs{\frac{\vtil^\b}{y(1-y)}\frac{dy}{d\x}-\frac{\e}{2(\t_1+\t_2)}\frac{(\g+1)}{\g\s_* p(v_-)}} \le C\e^2,
\]
which completes the proof.
\end{proof}

We are now ready to provide the proof of Proposition \ref{prop:main3}.

\begin{proof}[Proof of Proposition \ref{prop:main3}]
For given \((v_-,u_-)\), we choose positive constants \(\e_0\) and \(\d_3\) such that \(\e_0\) is smaller than the minimum values specified in Lemma \ref{lem-VS} and Lemma \ref{lem-Jacobian}, and
\[
\d_3 < \min \Big(\frac{1}{2},\frac{\d_*}{2}\Big)
\]
where the constant \(\d_*>0\) is defined in Lemma \ref{lem:local}.
In addition, we choose \(\e_0\) smaller than \(\d_3\). Now we rewrite the functionals \(\Ical_g^Y,\Ical_1,\Ical_2,\Gcal_p,\) and \(\Dcal_p\) with respect to $y$-variable given in \eqref{yandw}.

\noindent \(\bullet\) \textbf{Change of variable for \(\Ical_g^Y\):}
Recall from \eqref{note-in} that
\begin{align*}
\Ical_g^Y(v) &=
-\frac{1}{2\s_\e^2} \int_\RR a'\abs{p(v)-p(\vtil)}^2d\x
-\int_\RR a' Q(v|\vtil) d\x\\
&\qquad\qquad\qquad
-\int_\RR a p(\vtil)' (v-\vtil) d\x
+\frac{1}{\s_\e} \int_\RR a \htil' (p(v)-p(\vtil)) d\x.
\end{align*}
Then, using \eqref{Q-est-U}, \eqref{Q-est-L}, \eqref{ratio-vh} and \eqref{sm1}, we obtain
\begin{equation} \label{insidey}
\abs{\s_*^2 \frac{\l}{\e^2}\Ical_g^Y + \int_0^1 (W^2 + 2W) dy} \le C\d_3 \Big(\int_0^1 W^2 dy + \int_0^1 |W| dy\Big).
\end{equation}
By \eqref{insidey} and \eqref{assYp}, we note that 
\begin{align*}
\int_0^1 W^2 dy - 2 \abs{\int_0^1 W dy} &\le \abs{\int_0^1(W^2+2W) dy} \le \s_*^2\frac{\l}{\e^2}|\Ical^Y_g|+C \d_3 \Big(\int_0^1 W^2 dy + \int_0^1 |W|dy \Big)\\
&\le C_2 \s_*^2 + C \d_3 \Big(\int_0^1 W^2 dy + \int_0^1 |W|dy \Big).
\end{align*}
Then, since we have
\[
\abs{\int_0^1 W dy} \le \int_0^1 |W| dy \le \frac{1}{8} \int_0^1 W^2 dy + 8,
\]
we choose \(\d_3\) so that there exists a constant \(C_1\) (depends on \(C_2\), but not on \(\e,\l\)) such that
\begin{equation} \label{WL2C1}
\int_0^1 W^2 dy \le C_1.
\end{equation}
This, together with an inequality $-p^2\le -\frac{q^2}{2}+|p+q|^2$ and \eqref{insidey}, implies a precise estimate on \(\abs{\Ical_g^Y}^2\) as follows:
\begin{equation} \label{insidey2}
\begin{aligned}
- \s_*^2 \frac{\l^2}{\e^3} \frac{\abs{\Ical_g^Y}^2}{\e\d_3}
&= - \frac{1}{\s_*^2 \d_3} \abs{\s_*^2\frac{\l}{\e^2} \Ical_g^Y}^2 \\
&\le - \frac{1}{2\s_*^2\d_3} \abs{\int_0^1 (W^2+2W) dy}^2
+ C \d_3 \abs{\int_0^1 (W^2+|W|)dy}^2 \\
&\le - \frac{1}{2\s_*^2\d_3} \abs{\int_0^1 (W^2+2W) dy}^2
+ C \d_3 \int_0^1 W^2 dy.
\end{aligned}
\end{equation}

\vspace{2mm}
\noindent\(\bullet\) \textbf{Change of variable for \(\Ical_1-\Gcal_p\):}
Using \eqref{Q-est-U}--\eqref{Q-est-L} and \eqref{sm1}, we find that
\begin{equation} \label{insideb1g}
\begin{aligned}
\abs{2\a_\g \frac{\l^2}{\e^3}(\Ical_1-\Gcal_p) - \frac{2}{3}\int_0^1 W^3 dy}
&\le C \e \Big(\frac{\l}{\e}\Big) \int_0^1 W^2 dy + C \d_3 \int_0^1 \abs{W}^3 dy.
\end{aligned}
\end{equation}
Here, we recall \(\a_\g = \frac{\g\s_*p(v_-)}{\g+1}\).

\vspace{2mm}
\noindent\(\bullet\) \textbf{Change of variable for \(\Ical_2\):}
Using \eqref{p-est1}, \eqref{sm1}, \eqref{assYp} and \(\l\le\d_3\), we have
\begin{equation} \label{insideb2}
\begin{aligned}
\abs{2\a_\g \frac{\l^2}{\e^3} \Ical_2 - \int_0^1 W^2 dy} \le C\d_3 \int_0^1 W^2 dy.
\end{aligned}
\end{equation}

\vspace{2mm}
\noindent\(\bullet\) \textbf{Change of variable for \(\Dcal_p\):}
The diffusion term \(\Dcal_p\) can be rewritten as
\[
\Dcal_p = \t_1\frac{\e^2}{\l^2} \int_0^1 a v^\b \abs{W_y}^2 \Big(\frac{dy}{d\x}\Big) dy.
\]
Then, the Jacobian estimate \eqref{jacob} with \eqref{assYp} implies that
\begin{equation} \label{insided}
\abs{2\a_\g \frac{\l^2}{\e^3} \Dcal_p - \t_1\int_0^1 y(1-y) \abs{W_y}^2 dy}
\le C\d_3 \int_0^1 y(1-y) \abs{W_y}^2 dy.
\end{equation}

Since $\d\mapsto \Rcal_{\e,\d}(v)$ is an increasing function, we observe that for any \(\d\in(0,\d_3]\),
\begin{align*}
\Rcal_{\e,\d}(v) \le \Rcal_{\e,\d_3}(v)= 
- \frac{1}{\e \d_3} \abs{\Ical_g^Y(v)}^2 + \Ical_1(v)
+ \Ical_2(v)
-\Big(1-\d_3\frac{\e}{\l}\Big)\Gcal_p(v) - (1-\d_3)\Dcal_p(v).
\end{align*}
Combining \eqref{insidey2}--\eqref{insided}, it holds that for any \(\d \in (0,\d_3]\) and for some constants \(C_3,C_4>0\),
\begin{multline*}
2\a_\g \frac{\l^2}{\e^3}\Rcal_{\e,\d}(v) \le 
- \frac{1}{C_3 \d_3} \abs{\int_0^1 (W^2+2W) dy}^2 + \int_0^1 W^2 dy
+ C_4 \d_3 \int_0^1 W^2 dy \\
+ \frac{2}{3}\int_0^1 W^3 dy + C_4 \d_3 \int_0^1 |W|^3 dy 
- (0.9-C_4 \d_3) \int_0^1 y(1-y)|W_y|^2 dy,
\end{multline*}
where we used $\t_1\ge 0.9$ (see \eqref{0.9}) in the last inequality.

We now fix \(\d_2\) in Proposition \ref{prop_NPI} corresponding to \(C_1\) in \eqref{WL2C1}.
Then, we choose \(\d_3\) further small enough such that \(\max(C_3,C_4)\d_3 < \d_2\).
For such choice of $\d_3$, for any \(\e<\e_0\) and any \(\l\le\d_3\) with \(\e\le\l\), we apply Proposition \ref{prop_NPI} to establish the desired result \eqref{redelta}.
\end{proof}

\subsection{Truncation of large values of \(\abs{p(v)-p(\vtil)}\)}\label{sec:out}
To apply Proposition \ref{prop:main3} in the proof of Proposition \ref{prop:main}, two issues need to be addressed.
First, since the truncation size \(\d_3\) depends on the constant \(C_2\), it is necessary to obtain an estimate of \(C_2\) that is uniform with respect to the truncation size.
Second, we need to show that all bad terms for values of \(v\) such that \(\abs{p(v)-p(\vtil)}\ge\d_3\) have a negligible effect.
The following lemma provides the key ingredient for both tasks, namely the smallness of the localized relative entropy, from which the first one, i.e., a uniform estimate on \(C_2\), directly follows as in \eqref{lbisC2}.

To this end, we introduce a truncation on \(\abs{p(v)-p(\vtil)}\) with any truncation level \(k>0\) to estimate the constant \(C_2\).
Let \(\psi_k\) be a continuous function on \(\RR\) defined by 
\[
\psi_k(y) = \inf (k, \sup(-k,y)), \quad k>0.
\]
Then, we define the function \(\vbar_k\) uniquely (thanks to the monotonicity of $p$) by 
\[
p(\vbar_k)-p(\vtil) = \psi_k(p(v)-p(\vtil)).
\]
Note that the following properties of $\vbar_k$ hold:
\begin{equation}\label{vbar_prop}
	\norm{p(\vbar_k)-p(\vtil)}_{L^{\infty}(\RR)}\le k,\quad \mbox{and either $v\le \vbar_k\le \vtil$ or $\vtil \le \vbar_k\le v$}.
\end{equation}

We now state a lemma that provides the localized relative entropy smallness and a uniform estimate on the constant \(C_2\).

\begin{lemma} \label{lem:locE}
For a fixed state \((v_-,u_-)\in\RR^+\times\RR\), there exist constants \(C_2,k_0,\e_0,\d_0>0\) such that for any \(\e<\e_0\), \(\l<1/2\) with \(\e/\l<\d_0\), the following holds whenever \(\abs{Y(U)} \le \e^2\):\\
\begin{align}
&\int_\RR a' (h-\htil)^2 d\x
+ \int_\RR a' Q(v|\vtil) d\x \le C\frac{\e^2}{\l}, \label{locE}\\
&\hspace{10mm}
\abs{\Ical_g^Y(\vbar_k)} \le C_2 \frac{\e^2}{\l}, \quad \forall\, k \le k_0. \label{lbisC2}
\end{align}
\end{lemma}

\begin{proof}
Since \(Y\) has the same form as in the barotropic NS case, the proof is identical. Thus, we omit it and refer the reader to \cite[Lemma 4.1]{EEKO-ISO} and \cite[Lemma 4.4]{KV-Inven21}.
\end{proof}

We now fix the constant \(\d_3\), associated with the constant \(C_2\) from Lemma \ref{lem:locE}, for use in Proposition \ref{prop:main3}.
If it is necessary, we retake \(\d_3\) such that \(\d_3<k_0\) for the constant \(k_0\) in \eqref{lbisC2}.
Then, we set
\[
\vbar \coloneqq \vbar_{\d_3}, \quad \Ubar \coloneqq (\vbar,h).
\]
In what follows, we employ the following notation: 
\[
\O \coloneqq \{\x \mid (p(v)-p(\vtil))(\x) \le \d_3\}.
\]

To proceed further, we decompose \(\Bcal_{\d_3}\) in \eqref{badgood} into three parts in the following way:
\begin{equation} \label{B-decomp}
\Bcal_{\d_3} = B_1^- + B_1^+ + B_2
\end{equation}
where
\begin{align*}
B_1^- &\coloneqq \int_{\O^c} a' (p(v)-p(\vtil))(h-\htil) d\x, \qquad
B_1^+ \coloneqq \frac{1}{2\s_\e}\int_{\O} a' \abs{p(v)-p(\vtil)}^2 d\x, \\
B_2 &\coloneqq \s_\e\int_\RR a \vtil' p(v|\vtil) d\x.
\end{align*}
Moreover, we recall  \(\Gcal_h^-,\Gcal_h^+,\Gcal_p,\Dcal_h\) and \(\Dcal_p\) in \eqref{def-GD} for the good terms.

\vspace{2mm}
The key distinction between the two diffusion terms lies in the following monotonicity property. Precisely, \(\Dcal_p\) satisfies 
\begin{equation} \label{keyD}
\begin{aligned}
&\Dcal_p(U) = \t_1\int_\RR a v^\b \abs{\rd_\x (p(v)-p(\vtil))}^2 d\x \\
&= \t_1\int_\RR a v^\b \abs{\rd_\x (p(v)-p(\vtil))}^2 (\one{\abs{p(v)-p(\vtil)}\le\d_3}+\one{p(v)-p(\vtil)>\d_3}+\one{p(v)-p(\vtil)<-\d_3})d\x \\
&= \Dcal_p(\Ubar) + \t_1\int_\RR a v^\b \abs{\rd_\x (p(v)-p(\vtil))}^2 (\one{p(v)-p(\vtil)>\d_3}+\one{p(v)-p(\vtil)<-\d_3})d\x \ge \Dcal_p(\Ubar).
\end{aligned}
\end{equation}
However, \(\Dcal_h(U)\ge\Dcal_h(\Ubar)\) does not hold, and it remains invalid even if we define a truncation of \(h\) variable.
This constitutes the fundamental reason why \(\Dcal_h\) cannot be employed in Proposition \ref{prop:main3}.
On the other hand, \eqref{Phi-sim}, \eqref{vbar_prop} and \eqref{locE} imply the monotonicity property for the good term:
\begin{equation} \label{l2-p}
0 \le \s_\e \int_\RR a' \abs{Q(v|\vtil)-Q(\vbar|\vtil)}d\x
= \Gcal_p(U)-\Gcal_p(\Ubar) \le C\frac{\e^2}{\l}.
\end{equation}

\vspace{2mm}
The strategies for controlling \(B_1^-\) and \(B_1^+\) differ substantially and in particular, the control of \(B_2\) requires a delicate analysis.
In other words, our approach must distinguish between the cases where \(p(v)\) takes large values and where it takes small values.
For this purpose, we introduce the following one-sided truncation \(\vbar_s\) and \(\vbar_b\) as follows:
\begin{equation} \label{trunc-sb}
p(\vbar_s)-p(\vtil) \coloneqq \psi_{\d_3}^b(p(v)-p(\vtil)), \quad
p(\vbar_b)-p(\vtil) \coloneqq \psi_{\d_3}^s(p(v)-p(\vtil)),
\end{equation}
where \(\psi_{\d_3}^b\) and \(\psi_{\d_3}^s\) are one-sided truncations, i.e., 
\[
\psi_{\d_3}^b(y) \coloneqq \sup (-\d_3,y), \quad
\psi_{\d_3}^s(y) \coloneqq \inf (\d_3,y).
\]
Note that the truncation \(\vbar_s\) (resp. \(\vbar_b\)) represents the truncation of large (resp. small) values of \(v\) corresponding to \(\abs{p(v)-p(\vtil)}\ge\d_3\).
It is also worth noting that
\begin{align*}
    p(\vbar_s) - p(v)
&= (\psi_{\d_3}^b - I) (p(v)-p(\vtil))
= (-(p(v)-p(\vtil))-\d_3)_+, \\
p(v) - p(\vbar_b)
&= (I-\psi_{\d_3}^s)(p(v)-p(\vtil)) 
= ((p(v)-p(\vtil))-\d_3)_+, \\
\abs{p(v)-p(\vbar)}
&= \abs{(I-\psi_{\d_3})(p(v)-p(\vtil))} 
= (\abs{p(v)-p(\vtil)}-\d_3)_+.
\end{align*}
In addition, by \eqref{keyD} it holds that
\begin{equation} \label{eq-D}
\begin{aligned}
\Dcal_p(U) & \ge
\int_\RR a v^\b \abs{\rd_\x (p(v)-p(\vbar_s))}^2 d\x
+\int_\RR a v^\b \abs{\rd_\x (p(v)-p(\vbar_b))}^2 d\x.
\end{aligned}
\end{equation}

Furthermore, the following observations are required to proceed: for any \(\x\in\RR\) such that \(\abs{(p(v)-p(\vbar))(\x)}>0\) or equivalently \(\abs{(p(v)-p(\vtil))(\x)}>\d_3\), it holds that \(Q(v|\vtil)(\x)\ge\a\) for some constant \(\a>0\) depending only on \(\d_3\).
Thus,
\begin{equation} \label{Qconst}
\begin{aligned}
&\hspace{18mm}\one{\abs{p(v)-p(\vbar)}>0} = \one{\abs{p(v)-p(\vtil)}>\d_3}
\le \frac{Q(v|\vtil)}{\a},\\
&\one{\abs{p(v)-p(\vbar_s)}>0}+\one{\abs{p(v)-p(\vbar_b)}>0}
= \one{\abs{p(v)-p(\vtil)}>\d_3}
\le \frac{Q(v|\vtil)}{\a}.
\end{aligned}
\end{equation}
Moreover, since \(v\mapsto Q(v|\vtil)/v\) is increasing on \((\vtil,+\infty)\), we obtain
\begin{equation} \label{Q-lin}
v \one{\abs{p(v)-p(\vbar_s)}>0} = v \one{p(v)-p(\vtil)<-\d_3} \le C Q(v|\vtil).
\end{equation}

\vspace{2mm}
The following proposition provides the controls of the hyperbolic bad terms.
\begin{prop} \label{prop:hyp}
For any fixed constant state \(U_-\coloneqq(v_-,u_-)\in\RR^+\times\RR\), there exist constants \(\e_0,\d_0,C,C^* > 0\) (in particular, \(C\) and \(C^*\) depend on \(\d_3\)) such that for any \(\e<\e_0\) and \(\d_0^{-1}\e<\l<\d_0<1/2\) and any \(U\in\Hcal_T\) with \(\abs{Y(U)}\le\e^2\), the following holds.
\begin{align}
&|B_1^+(U)-\Ical_1(\vbar)| \le C \frac{\e}{\l} \Dcal_p(U) + C \frac{\e^4}{\l^2}\Gcal_p(U),
\label{bo1p}\\
&|B_1^-(U)| \le \d_3 \Gcal_h^-(U)
+ C \frac{\e}{\l} \Dcal_p(U) + C \frac{\e^4}{\l^2}\Gcal_p(U),
\label{bo1m}\\
&|B_2(U)-\Ical_2(\vbar)| \le C\frac{\e^2}{\l^2} \Dcal_p(U)
+C\frac{\e^5}{\l^3}\Gcal_p(U)
+C\frac{\e}{\l}(\Gcal_p(U)-\Gcal_p(\Ubar)),
\label{bo2}\\
&\begin{aligned}
&|\Bcal_{\d_3}(U)|+|B_1^+(U)|
\le C^* \Big(\frac{\e^2}{\l} + \frac{\e}{\l} \Dcal_p(U)\Big).
\end{aligned}
\label{hyp-total}
\end{align}
\end{prop}

\vspace{2mm}
Two key ingredients of our analysis are the localized smallness of the relative entropy and interpolation-type pointwise estimates.
The following lemma establishes the latter.

\begin{lemma} \label{lem-pw}
Under the same assumptions as Proposition \ref{prop:hyp}, there exists \(\x_0\in[-1/\e,1/\e]\) such that the following holds: for any \(\x\in\RR\),
\begin{align}
&\abs{(p(v)-p(\vbar_s))(\x)}
\le C \sqrt{\int_{\x_0}^\x Q(v|\vtil) d\z} \sqrt{\Dcal_p(U)},
\label{pwp1}\\
&\abs{\sqrt{v}(p(v)-p(\vbar_b))(\x)} 
\le C\Big(\sqrt{\int_{\x_0}^\x Q(v|\vtil) d\z}\sqrt{\Dcal_p(U)}
+ \sqrt{\frac{\e^4}{\l^3}} \sqrt{\Gcal_p(U)}\Big),
\label{pwp2}\\
&\begin{aligned}
&\abs{v(p(v)-p(\vbar_b))^2(\x)} 
\le C\sqrt{\int_{\x_0}^\x v (p(v)-p(\vbar_b))^2 d\z}\sqrt{\Dcal_p(U)}\\
&\hspace{60mm}
+ C\frac{\e^2}{\l^2}\Dcal_p(U) + C\frac{\e^5}{\l^3}\Gcal_p(U).
\end{aligned}
\label{pwp3}
\end{align}
\end{lemma}

\begin{proof}
Our proof is carried out in two steps.
\step{1} We first show that there exists a point \(\x_0\in[-1/\e,1/\e]\) such that \((p(v)-p(\vbar))(\x_0)=0\).\\
First of all, using \eqref{lower-v}, \eqref{der-scale} and \eqref{locE}, we find that
\[
2\e \int_{-1/\e}^{1/\e} Q(v|\vtil) d\x
\le \frac{2\e}{\inf_{[-1/\e,1/\e]}a'} \int_\RR a' Q(v|\vtil)d\x
\le C \frac{\e}{\e\l} \frac{\e^2}{\l} = C\Big(\frac{\e}{\l}\Big)^2.
\]
This shows that there exists \(\x_0\in[-1/\e,1/\e]\) such that \(Q(v|\vtil)(\x_0) \le C(\e/\l)^2\).
Thus, for sufficiently small \(\d_0>0\), we apply \eqref{p-quad} to find that
\[
\abs{(p(v)-p(\vtil))(\x_0)} \le C\frac{\e}{\l}.
\]
Then, for \(\d_0\) small enough to satisfy \(C\e/\l \le \d_3/2\), it follows from the definition of \(\vbar\) that
\[
\abs{(p(v)-p(\vbar))(\x_0)}=0.
\]
In addition, we also have
\[
\abs{(p(v)-p(\vbar_s))(\x_0)}=\abs{(p(v)-p(\vbar_b))(\x_0)}=0.
\]

\step{2} Next, we establish the pointwise estimates \eqref{pwp1}--\eqref{pwp3}. If $p(v)=p(\vbar_s)$, then, there is nothing to show. When $p(v)\neq p(\vbar_s)$, it follows from the definition of $\vbar_s$ that there is a positive lower bound of $v$. Then, the fundamental theorem of calculus and H\"older's inequality yield that for any \(\x\in\RR\),
\begin{align*}
\abs{p(v)-p(\vbar_s)(\x)}
&\le \abs{\int_{\x_0}^\x \rd_\z(p(v)-p(\vbar_s)) d\z}
\le C \abs{\int_{\x_0}^\x v^{\b/2} \rd_\z(p(v)-p(\vbar_s)) d\z} \\
&\le C \sqrt{\int_{\x_0}^\x \one{p(v)-p(\vtil)<-\d_3}d\z} \sqrt{\Dcal_p(U)}.
\end{align*}
Then, using \eqref{Qconst}, the first pointwise estimate \eqref{pwp1} can be established.

Next, we show \eqref{pwp2} and \eqref{pwp3}. Again, if $p(v)=p(\vbar_b)$, then there is nothing to show. On \(\{p(v)-p(\vbar_b) \neq 0\}\), there is no strictly positive lower bound for \(v\), and therefore the proofs of the remaining estimates are more delicate.
Since \(v\) has an upper bound on \(\{p(v)-p(\vbar_b) \neq 0\}\), we apply \eqref{locE} and \eqref{Qconst} to find that for any \(\x\in\RR\),
\begin{equation} \label{pw-b}
\begin{aligned}
&\abs{\sqrt{v} (p(v)-p(\vbar_b))(\x)} 
\le C \abs{\int_{\x_0}^\x \sqrt{v} \rd_\z(p(v)-p(\vbar_b)) d\z}\\
&+ C \abs{\int_{\x_0}^\x \frac{(p(v)-p(\vbar_b))}{p(v)} \sqrt{v} \rd_\z(p(v)-p(\vtil)) d\z}
+ C \abs{\int_{\x_0}^\x \frac{(p(v)-p(\vbar_b))}{p(v)} \sqrt{v} p(\vtil)' d\z} \\
&\le C\sqrt{\int_{\x_0}^\x v^{1-\b}\one{p(v)-p(\vtil)>\d_3} d\z}\sqrt{\Dcal_p(U)}
+ C\abs{\int_{\x_0}^\x \vtil' \one{p(v)-p(\vtil)>\d_3} d\z} \\
&\le C\sqrt{\int_{\x_0}^\x Q(v|\vtil) d\z}\sqrt{\Dcal_p(U)}
+ C \sqrt{\frac{\e^4}{\l^3}} \sqrt{\Gcal_p(U)}.
\end{aligned}
\end{equation}

To prove \eqref{pwp3}, we proceed in a similar way as follows:
\begin{equation} \label{pw-bb}
\begin{aligned}
&\abs{v (p(v)-p(\vbar_b))^2(\x)} 
\le C \abs{\int_{\x_0}^\x v (p(v)-p(\vbar_b)) \rd_\z(p(v)-p(\vbar_b)) d\z}\\
&+ C \abs{\int_{\x_0}^\x \frac{(p(v)-p(\vbar_b))^2}{p(v)} v \rd_\z(p(v)-p(\vtil)) d\z}
+ C \abs{\int_{\x_0}^\x \frac{(p(v)-p(\vbar_b))^2}{p(v)} v p(\vtil)' d\z} \\
&\le C\sqrt{\int_{\x_0}^\x v (p(v)-p(\vbar_b))^2 d\z}\sqrt{\Dcal_p(U)}
+ C\underbrace{\abs{\int_\RR \abs{\vtil'} v (p(v)-p(\vbar_b))^2 d\z}}_{\eqqcolon J}. \\
\end{aligned}
\end{equation}
Then, for \(J\), we utilize \eqref{pwp2}, \eqref{locE}, \eqref{der-scale}, and Lemma \ref{lemma_pushing} as follows:
\begin{equation} \label{vp2-c}
\begin{aligned}
J &\le
C\Dcal_p(U) \int_\RR \abs{\vtil'(\x)} \int_{\x_0}^\x Q(v|\vtil) d\z d\x
+C\frac{\e^4}{\l^3}\Gcal_p(U)\int_\RR \abs{\vtil'(\x)} d\x\\
&\le
C\Dcal_p(U)\frac{1}{\e} \int_\RR \abs{\vtil'} Q(v|\vtil) d\x
+C\frac{\e^5}{\l^3}\Gcal_p(U)
\le C\frac{\e^2}{\l^2}\Dcal_p(U)+C\frac{\e^5}{\l^3}\Gcal_p(U).
\end{aligned}
\end{equation}
This completes the proof of \eqref{pwp3}.
\end{proof}

We note that \eqref{pwp3} constitutes a new type of pointwise estimate, which, in conjunction with Lemma \ref{lemma_pushing}, enables the exploitation of the exponential decay of the shock derivatives.

\vspace{2mm}
We are now ready to prove Proposition \ref{prop:hyp}.

\begin{proof}[Proof of Proposition \ref{prop:hyp}]
\bpf{bo1p}
We consider the following decomposition:
\begin{align*}
&\abs{B_1^+(U)-\Ical_1(\vbar)}
=\frac{1}{2\s_\e} \abs{\int_\O a'\abs{p(v)-p(\vtil)}^2d\x - \int_\RR a'\abs{p(\vbar)-p(\vtil)}^2d\x}\\
&\le
C \int_\O a' \abs{p(v)-p(\vbar)}^2 d\x
+ C \int_\O a' \abs{p(v)-p(\vbar)} d\x
+ C \int_{\O^c} a' d\x
\eqqcolon J_{11}+J_{12}+J_{13}.
\end{align*}
Since \(\vbar=\vbar_s\) on \(\O\), it follows from  \eqref{pwp1}, Lemma \ref{lemma_pushing} and \eqref{locE} that
\begin{equation} \label{bo1pJ1}
J_{11} \le C\Dcal_p(U) \int_\RR a' \int_{\x_0}^\x Q(v|\vtil) d\z d\x
\le C\Dcal_p(U) \frac{1}{\e} \int_\RR a' Q(v|\vtil) d\x
\le C\frac{\e}{\l}\Dcal_p(U).
\end{equation}
To handle \(J_{12}\), we carry out a nonlinearization so as to obtain a quadratic structure in the following way: we temporarily employ the truncations of \(v\) by \(\d_3/2\) in place of \(\d_3\), which we denote by \((\vbar_s)_{\d_3/2}\) and \((\vbar_b)_{\d_3/2}\), namely, 
\begin{equation} \label{half-trunc}
p((\vbar_s)_{\d_3/2})-p(\vtil) \coloneqq \psi_{\d_3/2}^b (p(v)-p(\vtil)), \quad
p((\vbar_b)_{\d_3/2})-p(\vtil) \coloneqq \psi_{\d_3/2}^s (p(v)-p(\vtil)).
\end{equation}
It is important to observe that all the estimates in Lemma \ref{lem-pw} hold for \((\vbar_s)_{\d_3/2}\) and \((\vbar_b)_{\d_3/2}\) as well as for \(\vbar_s\) and \(\vbar_b\).
Then, since \((y-\d_3) \ge 0\) implies \((y-\d_3/2)_+ \ge \d_3/2\), it follows that
\begin{equation} \label{NL}
(y-\d_3)_+ \le (y-\d_3/2)_+ \one{(y-\d_3) \ge 0}
\le (y-\d_3/2)_+ \left(\frac{(y-\d_3/2)_+}{\d_3/2}\right)
\le \frac{2}{\d_3} ((y-\d_3/2)_+)^2.
\end{equation}
This yields the following inequality encapsulating the nonlinearization:
\[
\abs{p(v)-p(\vbar_s)} \one{p(v)-p(\vtil) < -\d_3} \le C |p(v)-p((\vbar_s)_{\d_3/2})|^2,
\]
where $C$ depends on $\delta_3$. Thus, using the same argument as \eqref{bo1pJ1}, we obtain
\begin{equation} \label{bo1pJ2}
J_{12} \le C \int_\O a' \abs{p(v)-p(\vbar)} d\x
\le C \int_\O a' |p(v)-p((\vbar)_{\d_3/2})|^2 d\x
\le C \frac{\e}{\l} \Dcal_p(U).
\end{equation}
It remains to control \(J_{13}\).
Since both \(|p(v)-p((\vbar_b)_{\d_3/2})|\) and \(v|p(v)-p((\vbar_b)_{\d_3/2})|\) are strictly bounded away from zero on \(\O^c\), we have
\begin{equation} \label{NL-J3}
\mathbf{1}_{\O^c} = \one{p(v)-p(\vtil)>\d_3}
\le C v |p(v)-p((\vbar_b)_{\d_3/2})|^2,
\end{equation}
for some constant \(C\) that depends on \(\d_3\).
Then, using \eqref{vp2-c} with \eqref{der-scale}, we find that
\begin{equation} \label{bo1pJ3}
J_{13}
\le C \int_{\O^c} a' d\x
\le C \int_\RR a' v |p(v)-p((\vbar_b)_{\d_3/2})|^2 d\x
\le C \frac{\e}{\l} \Dcal_p(U) + C \frac{\e^4}{\l^2}\Gcal_p(U).
\end{equation}
Gathering \eqref{bo1pJ1}, \eqref{bo1pJ2}, and \eqref{bo1pJ3}, the desired result \eqref{bo1p} can be established.

\vspace{2mm}
\bpf{bo1m}
We consider the following decomposition:
\begin{align*}
\abs{B_1^-(U)} &\le
\abs{\int_{\O^c} a' (p(v)-p(\vtil))(h-\htil) \one{p(v)-p(\vtil) \le 2\d_3} d\x} \\
&\qquad
+\abs{\int_{\O^c} a' (p(v)-p(\vbar_b))(h-\htil) \one{p(v)-p(\vtil) > 2\d_3} d\x} \\
&\qquad
+\abs{\int_{\O^c} a' (p(\vbar_b)-p(\vtil))(h-\htil) \one{p(v)-p(\vtil) > 2\d_3} d\x}
\eqqcolon J_{21}+J_{22}+J_{23}.
\end{align*}
Thanks to the definition of \(\O\) and \(\Gcal_h^-\), it simply follows that 
\[
J_{21}+J_{23} \le 3\d_3 \abs{\int_{\O^c} a' (h-\htil)d\x}
\le \d_3 \Gcal_h^-(U)
+C\int_{\O^c} a' d\x.
\]
Then, using \eqref{bo1pJ3}, we obtain 
\[
J_{21}+J_{23} \le \d_3 \Gcal_h^-(U)
+ C \frac{\e}{\l} \Dcal_p(U) + C \frac{\e^4}{\l^2}\Gcal_p(U).
\]

What remains is to control \(J_{22}\), which constitutes one of the most delicate aspects of our analysis.
By an argument analogous to that in \eqref{NL-J3}, we obtain
\begin{equation} \label{powerup}
(p(v)-p(\vbar_b)) \one{p(v)-p(\vtil) > 2\d_3}
\le C v(p(v)-p(\vbar_b))^2.
\end{equation}
This together with \eqref{locE} implies that
\begin{align*}
J_{22} &\le \Big(\int_{\O^c} a' (p(v)-p(\vbar_b))^2 \one{p(v)-p(\vtil) > 2\d_3} d\x\Big)^{1/2}
\Big(\int_{\O^c} a' (h-\htil)^2 d\x\Big)^{1/2} \\
&\le C \sqrt{\frac{\e^2}{\l}} \Big(\int_\RR a' v^2(p(v)-p(\vbar_b))^4 d\x\Big)^{1/2}.
\end{align*}
Then, we apply \eqref{pwp3} in the following way:
\begin{align*}
\int_\RR a' v^2(p(v)-p(\vbar_b))^4 d\x
&\le C\Dcal_p(U) \int_\RR a' \int_{\x_0}^\x v(p(v)-p(\vbar_b))^2 d\z d\x\\
&\qquad
+C \Big(\frac{\e^4}{\l^4}\Dcal_p(U)^2
+\frac{\e^{10}}{\l^6}\Gcal_p(U)^2\Big)
\int_\RR a' \one{p(v)-p(\vtil) > \d_3} d\x.
\end{align*}
Using Lemma \ref{lemma_pushing}, \eqref{Qconst} and \eqref{locE}, we obtain
\begin{align*}
\int_\RR a' v^2(p(v)-p(\vbar_b))^4 d\x
&\le \frac{C}{\e}\Dcal_p(U)\int_\RR a' v(p(v)-p(\vbar_b))^2 d\x
+C \Big(\frac{\e^6}{\l^5}\Dcal_p(U)^2
+\frac{\e^{12}}{\l^7}\Gcal_p(U)^2\Big).
\end{align*}
We apply \eqref{bo1pJ3} and Young's inequality for the first term on the right-hand side to obtain
\begin{align*}
&\int_\RR a' v^2(p(v)-p(\vbar_b))^4 d\x
\le \frac{C}{\l} \Dcal_p(U)^2
+C\frac{\e^6}{\l^3}\Gcal_p(U)^2.
\end{align*}
Thus, since both \(\Dcal_p(U)\) and \(\Gcal_p(U)\) are positive, it follows that
\[
J_{22} \le C \Big(\frac{\e^2}{\l^2} \Dcal_p(U)^2 + \frac{\e^8}{\l^4}\Gcal_p(U)^2 \Big)^{1/2}
\le C\frac{\e}{\l}\Dcal_p(U) + C\frac{\e^4}{\l^2}\Gcal_p(U).
\]
Therefore, the desired result \eqref{bo1m} can be established.

\vspace{2mm}
\bpf{bo2} Proceeding as in \cite[Lemma 3.3]{KV-JEMS21} and \cite[Proposition 4.3]{EEKO-ISO}, we arrive at
\begin{align*}
\abs{B_2(U)-\Ical_2(\vbar)}
&\le C\int_\RR \vtil' \abs{p(v|\vtil)-p(\vbar|\vtil)} d\x\\
&\le C \int_\RR \vtil' \abs{p(v)-p(\vbar)} d\x
+ C \int_\RR \vtil' (Q(v|\vtil)-Q(\vbar|\vtil)) d\x
\eqqcolon J_{31}+J_{32}.
\end{align*}
To handle \(J_{31}\), we consider the following decomposition:
\begin{align*}
J_{31}
&= C \int_\RR \vtil' \abs{p(v)-p(\vbar_s)} d\x
+ C \int_\RR \vtil' \abs{p(v)-p(\vbar_b)} \one{p(v)-p(\vtil) > 2\d_3} d\x \\
&\qquad
+ C \int_\RR \vtil' \abs{p(v)-p(\vbar_b)} \one{p(v)-p(\vtil)\le 2\d_3} d\x
\eqqcolon J_{311}+J_{312}+J_{313}.
\end{align*}
Thanks to the nonlinearization argument in the proof of \eqref{bo1p}, using \eqref{der-scale} yields that
\[
J_{311} \le C \frac{\e^2}{\l^2} \Dcal_p(U).
\]
Then, following \eqref{powerup} and \eqref{vp2-c}, we get
\begin{align*}
J_{312} \le C \int_\RR \abs{\vtil'} v \abs{p(v)-p(\vbar_b)}^2 d\x
&\le C\frac{\e^2}{\l^2} \Dcal_p(U)
+C\frac{\e^5}{\l^3}\Gcal_p(U).
\end{align*}
Moreover, using the definition of \(\vbar_b\) with \eqref{der-scale}, we apply \eqref{bo1pJ3} to find that
\[
J_{313} \le C\frac{\e}{\l} \int_\RR a' \one{p(v)-p(\vtil)> \d_3} d\x
\le C\frac{\e^2}{\l^2}\Dcal_p(U) + C\frac{\e^5}{\l^3}\Gcal_p(U).
\]
On the other hand, for \(J_{32}\), it simply holds by \eqref{der-scale} that 
\[
J_{32} \le C\frac{\e}{\l}(\Gcal_p(U)-\Gcal_p(\Ubar)).
\]
Thus, summing up all, we obtain \eqref{bo2} as we desired.

\vspace{2mm}
\bpf{hyp-total} Using the properties of the relative functionals, namely, \eqref{p-est1}, \eqref{p-quad} and \eqref{Phi-sim}, we apply \eqref{locE} to find that 
\[
\abs{\Ical_1(\vbar)} + \abs{\Ical_2(\vbar)}
\le C \int_\RR a' Q(\vbar|\vtil) d\x
\le C \int_\RR a' Q(v|\vtil) d\x
\le C \frac{\e^2}{\l}.
\]
Then, \eqref{hyp-total} directly follows from \eqref{bo1p}--\eqref{bo2}, \eqref{locE} and \eqref{l2-p}.
\end{proof}

\subsection{Estimates on the parabolic bad terms} \label{subsec:P}
This subsection is devoted to controlling the parabolic bad terms.
Although these terms can be controlled by a small portion of the good terms owing to the exponential decay and smallness of the shock derivative, handling large perturbations requires a highly delicate analysis.
To this end, we first introduce the following notation:
\begin{equation*} \label{para-decomp-1}
\Pcal_1 = B_3 + B_4 + B_5
\end{equation*}
where
\begin{align*}
B_3 &\coloneqq -\t_1\int_\RR a' v^\b(p(v)-p(\vtil))\rd_\x(p(v)-p(\vtil)) d\x, \\
B_4 &\coloneqq -\t_1\int_\RR a' (p(v)-p(\vtil))(v^\b-\vtil^\b)\rd_\x p(\vtil) d\x, \\
B_5 &\coloneqq -\t_1\int_\RR a \rd_\x(p(v)-p(\vtil))(v^\b-\vtil^\b)\rd_\x p(\vtil) d\x.
\end{align*}
Recall that \(\b\coloneqq\g-\a\in[0,1]\).
In addition, we introduce the following notation:
\begin{equation*} \label{para-decomp-2}
\Pcal_2 = B_6 + B_7 + B_8
\end{equation*}
where
\begin{align*}
B_6 &\coloneqq -\t_2\g\int_\RR a' \frac{1}{v^{\a+1}} (h-\htil) \rd_\x(h-\htil) d\x, \\
B_7 &\coloneqq \t_2\g\int_\RR a(h-\htil)\Big(\frac{1}{v^{\a+1}}-\frac{1}{\vtil^{\a+1}}\Big) \htil'' d\x, \\
B_8 &\coloneqq \t_2\g\int_\RR a (h-\htil) \rd_\x\Big(\frac{1}{v^{\a+1}}-\frac{1}{\vtil^{\a+1}}\Big) \htil' d\x.
\end{align*}

The following proposition provides the controls of the parabolic bad terms in \(\Pcal_1\).
\begin{prop} \label{prop:para-1}
Under the same assumptions as Proposition \ref{prop:hyp}, there exists a positive constant $C^*$ such that 
\begin{align}
&\hspace{10mm}\abs{B_3(U)} \le C\l\Dcal_p(U) + C\e\Gcal_p(U),
\label{para-B3}\\
&\hspace{10mm}\abs{B_4(U)} \le C\frac{\e^3}{\l}\Dcal_p(U) + C\e^2\Gcal_p(U),
\label{para-B4}\\
&\hspace{10mm}\abs{B_5(U)} \le 
C\frac{\e^2}{\l^2}\Dcal_p(U) + C\e\l\Gcal_p(U),
\label{para-B5}\\
\abs{\Pcal_1(U)} &\le
C \Big(\l + \frac{\e^2}{\l^2}\Big)\Dcal_p(U) + C\e\Gcal_p(U)
\le C^*\frac{\e^2}{\l} + C \Big(\l + \frac{\e^2}{\l^2}\Big)\Dcal_p(U).
\label{para-1-total}
\end{align}
\end{prop}

\begin{proof}
Because of the shared structure with the barotropic NS system, the parabolic bad terms in \(\Pcal_1\) coincide with those that arise in the NS case.
Hence, using the pointwise estimate
\begin{equation} \label{pwp4}
\begin{aligned}
\abs{v^{\b/2} (p(v)-p(\vbar_b))(\x)}
&\le C\sqrt{\int_{\x_0}^\x Q(v|\vtil) d\z}\sqrt{\Dcal_p(U)}
+ C \sqrt{\frac{\e^4}{\l^3}} \sqrt{\Gcal_p(U)}, \quad \forall\,\x\in\RR,
\end{aligned}
\end{equation}
which can be established in the same way as \eqref{pwp2}, the argument follows as in \cite[Proposition 4.3]{KV-Inven21} and \cite[Proposition 4.4]{EEKO-ISO}.
Thus, we omit the details.
\end{proof}

The following proposition provides the controls of the parabolic bad terms in \(\Pcal_2\).
\begin{prop} \label{prop:para-2}
Under the same assumptions as Proposition \ref{prop:hyp}, there exists a positive constant $C^*$ such that
\begin{align}
\abs{B_6(U)} &\le
C\Big(\l\Dcal_h(U)+\frac{\e^2}{\l}\Dcal_p(U)\Big)
+ C \e\Big(\Gcal_h^-(U)+\Gcal_h^+(U)+\Gcal_p(U)\Big),
\label{para-B6}\\
\abs{B_7(U)} &\le C\frac{\e^3}{\l^2}\Big(\Dcal_p(U)+\Dcal_h(U)\Big)
+C\frac{\e^2}{\l}\Big(\Gcal_h^-(U)+\Gcal_h^+(U)+\Gcal_p(U)\Big),
\label{para-B7}\\
\abs{B_8(U)} &\le C\frac{\e}{\l} \Dcal_p(U) + C\e^2 \Big(\Gcal_h^-(U)+\Gcal_h^+(U)+\Gcal_p(U)\Big),
\label{para-B8}\\
\abs{\Pcal_2(U)} &\le
C\Big(\l\Dcal_h(U)+\frac{\e}{\l}\Dcal_p(U)\Big)
+ C \e\Big(\Gcal_h^-(U)+\Gcal_h^+(U)+\Gcal_p(U)\Big),
\label{para-2-total}\\
\abs{\Pcal_2(U)} &\le
C^*\frac{\e^2}{\l}
+C\Big(\l\Dcal_h(U)+\frac{\e}{\l}\Dcal_p(U)\Big).
\label{para-2-total-0}
\end{align}
\end{prop}

To prove the proposition above, it is necessary to define the truncation of \(h\)-variable:
\begin{equation} \label{ubar}
\hbar-\htil = \psi_{\d_3}(h-\htil).
\end{equation}
Notice that the size of truncation does not have to be \(\d_3\) and it is unnecessary for \(h\)-variable to introduce one-sided truncations \(\hbar_s\) and \(\hbar_b\).
(Moreover, \(\Ubar\) still stands for \((\vbar,h)\), not \((\vbar,\hbar)\).)
We also have
\[
\Dcal_h(U) = \t_2\g\int_\RR a \frac{1}{v^{\a+1}} |\rd_\x (h-\htil)|^2 d\x
\ge \t_2\g\int_\RR a \frac{1}{v^{\a+1}} |\rd_\x (h-\hbar)|^2 d\x,
\]
but it does not holds that
\[
\Dcal_h(U) \ge \t_2\g\int_\RR a \frac{1}{\vbar^{\a+1}}\abs{\rd_\x (h-\hbar)}^2 d\x.
\]
This is the reason why we cannot use \(\Dcal_h(U)\) in Proposition \ref{prop:main3}, as discussed earlier.

We now state the following lemma, which provides estimates that are needed in the proof of Proposition \ref{prop:para-2}.

\begin{lemma} \label{lem-h-con}
Under the same assumptions as Proposition \ref{prop:hyp}, there exists \(\x_0\in[-1/\e,1/\e]\) such that for any \(\x\in\RR\),
\begin{equation} \label{pwh}
\begin{aligned}
&\abs{v^{\frac{\b}{2}-1}(h-\hbar)(\x)} 
\le \sqrt{\Dcal_h(U)}\sqrt{\int_{\x_0}^\x v^{\g-1}\one{|h-\htil|>\d_3} d\z} \\
&+ C\sqrt{\Dcal_p(U)}\sqrt{\int_{\x_0}^\x v^{2\g-2} |h-\hbar| d\z}
+ C\sqrt{\frac{\e^4}{\l^3}}\Big(\sqrt{\int_\RR a'(h-\htil)^2 d\x} + \sqrt{\Gcal_p(U)}\Big).
\end{aligned}
\end{equation}
Moreover, we have
\begin{equation} \label{u-aux}
\int_\RR a'(h-\htil)^2 d\x
\le C\big(\Gcal_h^-(U)+\Gcal_h^+(U)+\Gcal_p(U)+\frac{\e}{\l}\Dcal_p(U)\big).
\end{equation}
\end{lemma}

\begin{proof}
By the same argument as in Lemma \ref{lem-pw}, we choose \(\x_0\in[-1/\e,1/\e]\) such that 
\[
\abs{(h-\hbar)(\x_0)}=0.
\]
Moreover, the same principle as in \eqref{Qconst} implies
\begin{equation} \label{Qconst-h}
\one{|h-\htil|>\d_3}=\one{|h-\hbar|>0} \le \frac{(h-\htil)^2}{\a'},
\end{equation}
for some constant \(\a'>0\) depending on \(\d_3\).
Then, using the fundamental theorem of calculus, H\"older's inequality, Young's inequality and \eqref{locE}, we find that for any \(\x\in\RR\),
\begin{align*}
\abs{v^{\frac{\b}{2}-1}(h-\hbar)(\x)}
&\le \int_{\x_0}^\x v^{\frac{\b}{2}-1} |(h-\hbar)_\z| d\z
+C\int_{\x_0}^\x v^{\g+\frac{\b}{2}-1} \abs{(p(v)-p(\vtil))_\z} |h-\hbar| d\z\\
&\qquad
+C\int_{\x_0}^\x v^{\g+\frac{\b}{2}-1} |h-\hbar| \abs{p(\vtil)'} d\z\\
&\le \sqrt{\Dcal_h(U)} \sqrt{\int_{\x_0}^\x v^{\g-1}\one{|h-\htil|>\d_3}d\z}
+\sqrt{\Dcal_p(U)} \sqrt{\int_{\x_0}^\x v^{2\g-2}(h-\hbar)^2 d\z}\\
&\qquad
+C\sqrt{\frac{\e^4}{\l^3}}\sqrt{\int_\RR a'(h-\htil)^2 d\x}
+C\frac{\e}{\l}\int_\RR a' v^{2\g+\b-2}\one{|h-\htil|>\d_3}d\x.
\end{align*}
Then, since \(\g\le1+\frac{\a}{3}\) and \(2\g+\b-2 \le 1\), we apply \eqref{Q-lin}, \eqref{Qconst-h} and \eqref{locE} to find that
\begin{align*}
\int_\RR a' v^{2\g+\b-2}\one{|h-\htil|>\d_3}d\x
&\le C \int_\RR a' \big(Q(v|\vtil)+(h-\htil)^2\big) d\x \\
&\le C \sqrt{\frac{\e^2}{\l}}\sqrt{\int_\RR a' \big(Q(v|\vtil)+(h-\htil)^2\big) d\x}.
\end{align*}
Thus, summing up all, \eqref{pwh} can be established.

To prove \eqref{u-aux}, we first recall that forms of the hyperbolic good terms are different on the sets \(\O\) and \(\O^c\).
Then, Young's inequality implies 
\begin{align*}
\int_\O a'(h-\htil)^2 d\x
&\le C\Gcal_h^+(U) + C\int_\O a'(p(v)-p(\vtil))^2 d\x \\
&\le C\Gcal_h^+(U) + C\int_\O a'(p(v)-p(\vbar_s))^2 d\x
+C\int_\O a'(p(\vbar_s)-p(\vtil))^2 d\x.
\end{align*}
Then, the locally quadratic structure of the relative entropy yields that
\[
\int_\O a'(p(\vbar_s)-p(\vtil))^2 d\x \le C\int_\RR a'Q(\vbar|\vtil)d\x
\le C\int_\RR a'Q(v|\vtil)d\x \le C\Gcal_p(U).
\]
Moreover, we apply \eqref{pwp1} with Lemma \ref{lemma_pushing} and use \eqref{Qconst} and \eqref{locE} to control the second term on the right-hand side: 
\[
\int_\O a'(p(v)-p(\vbar_s))^2 d\x
\le C \Dcal_p(U) \int_\RR a' \int_{\x_0}^\x Q(v|\vtil) d\z d\x
\le C\frac{\e}{\l}\Dcal_p(U).
\]
Combining the above estimates, we get \eqref{u-aux}.
\end{proof}

Now we are ready to prove Proposition \ref{prop:para-2}.

\begin{proof}[Proof of Proposition \ref{prop:para-2}]
The initial step is to derive an \(L^\infty\) bound for \(v^{-\g-\a}\).
To this end, using \eqref{pwp4} and \eqref{locE}, we observe that
\begin{align*}
v^{-\frac{\g+\a}{2}}(\x)
&=v^{\frac{\b}{2}}p(v)(\one{p(v)-p(\vtil)>\d_3}+\one{p(v)-p(\vtil)\le\d_3})(\x)\\
&\le v^{\frac{\b}{2}}(p(v)-p(\vbar_b))(\x)+C
\le C\sqrt{\Dcal_p(U)}\sqrt{\int_{\x_0}^\x Q(v|\vtil)d\z} + C.
\end{align*}
Then, from Lemma \ref{lemma_Linfty} with \eqref{Qconst} and \eqref{locE}, we get 
\begin{equation} \label{Linfty-O}
a'(\x) v^{-(\g+\a)}(\x) \le C\Dcal_p(U)\int_\RR a' Q(v|\vtil)d\x + C\e\l
\le C\frac{\e^2}{\l}\Dcal_p(U) + C\e\l.
\end{equation}
Moreover, in the sense of interpolation, we have 
\begin{equation} \label{Linfty-G}
a'(\x) v^{q}(\x)
\le C\frac{\e^2}{\l}\Dcal_p(U) + C\e\l, \quad \forall \, q\in [-(\g+\a),0].
\end{equation}
Note that \eqref{Linfty-O} and \eqref{Linfty-G} will be useful later.

\vspace{2mm}
\bpf{para-B6}
First of all, we apply Young's inequality as follows:
\[
\abs{B_6(U)}
\le C\l \int_\RR a \frac{1}{v^{\a+1}} (\rd_\x(h-\htil))^2 d\x
+\frac{C}{\l} \int_\RR |a'|^2 \frac{1}{v^{\a+1}} (h-\htil)^2 d\x.
\]
Then, since \(-(\a+1) \in [-(\g+\a),0]\), we apply \eqref{Linfty-G} with \eqref{locE} and \eqref{u-aux} to find that 
\begin{align*}
\frac{C}{\l} \int_\RR |a'|^2 \frac{1}{v^{\a+1}} (h-\htil)^2 d\x
&\le C\frac{\e^4}{\l^3}\Dcal_p(U)
+ C \e \int_\RR a' (h-\htil)^2 d\x\\
&\le C\frac{\e^2}{\l}\Dcal_p(U)
+ C \e\Big(\Gcal_h^-(U)+\Gcal_h^+(U)+\Gcal_p(U)\Big).
\end{align*}
Thus, \eqref{para-B6} is proved.

\vspace{2mm}
\bpf{para-B7}
We consider the following decomposition:
\begin{align*}
B_7(U)&= \t_2\g\int_\RR a(h-\htil)\Big(\frac{1}{v^{\a+1}}-\frac{1}{\vtil^{\a+1}}\Big) \htil'' 
\one{p(v)-p(\vtil)>\d_3}d\x\\
&\qquad+ \t_2\g\int_\RR a(h-\htil)\Big(\frac{1}{v^{\a+1}}-\frac{1}{\vtil^{\a+1}}\Big) \htil'' 
\one{\abs{p(v)-p(\vtil)}\le\d_3}d\x\\
&\qquad+ \t_2\g\int_\RR a(h-\htil)\Big(\frac{1}{v^{\a+1}}-\frac{1}{\vtil^{\a+1}}\Big) \htil'' 
\one{p(v)-p(\vtil)<-\d_3}d\x
\eqqcolon J_{41}+J_{42}+J_{43}.
\end{align*}
To handle \(J_{41}\), we observe
\begin{align*}
&\abs{(h-\htil)\Big(\frac{1}{v^{\a+1}}-\frac{1}{\vtil^{\a+1}}\Big)}\one{p(v)-p(\vtil)>\d_3}
\le  
\frac{C}{v^{\a+1}}(|h-\hbar|+|\hbar-\htil|)\one{p(v)-p(\vtil)>\d_3}\\
&\le C|h-\hbar|v^{\frac{\b}{2}-1}v^\frac{\b}{2}p(v)\one{p(v)-p(\vtil)>\d_3}
+\frac{C}{v^{\a+1}}\one{p(v)-p(\vtil)>\d_3}
\eqqcolon K_1+K_2.
\end{align*}
Then, we have
\begin{align*}
K_1 &\le C|h-\hbar|v^{\frac{\b}{2}-1}v^\frac{\b}{2}(p(v)-p(\vbar_b))\one{p(v)-p(\vtil)>\d_3}
+C|h-\hbar|v^{\frac{\b}{2}-1}v^\frac{\b}{2}p(\vbar_b)\one{p(v)-p(\vtil)>\d_3}\\
&\le C(v^{\frac{\b}{2}-1}(h-\hbar))^2\one{p(v)-p(\vtil)>\d_3}
+C(v^\frac{\b}{2}(p(v)-p(\vbar_b)))^2
+C\one{p(v)-p(\vtil)>\d_3}.
\end{align*}
Moreover, it holds that
\begin{align*}
K_2 
&\le C(v^\frac{\b}{2}(p(v)))^2 \one{p(v)-p(\vtil)>\d_3}
\le C(v^\frac{\b}{2}(p(v)-p(\vbar_b)))^2 
+C\one{p(v)-p(\vtil)>\d_3}.
\end{align*}
Thanks to \eqref{2nd-der} and \eqref{der-scale}, we have \(|\htil''| \le C\frac{\e^2}{\l}a'\) and thus we gather all the inequalities above to find that
\begin{align*}
\abs{J_{41}} &\le 
C\frac{\e^2}{\l} \int_\RR a' \one{p(v)-p(\vtil)>\d_3} d\x
+C\frac{\e^2}{\l}\int_\RR a' (v^\frac{\b}{2}(p(v)-p(\vbar_b)))^2 d\x\\
&\qquad
+C\frac{\e^2}{\l}\int_\RR a' (v^{\frac{\b}{2}-1}(h-\hbar))^2\one{p(v)-p(\vtil)>\d_3} d\x
\eqqcolon J_{411}+J_{412}+J_{413}.
\end{align*}
We analyze the right-hand side term-by-term basis.
For \(J_{11}\), we make use of \eqref{Qconst} to obtain 
\[
J_{411}
\le C\frac{\e^2}{\l}\Gcal_p(U).
\]
Thanks to \eqref{pwp4}, we use Lemma \ref{lemma_pushing} with \eqref{locE} and \eqref{Qconst} to control \(J_{412}\) as
\begin{align*}
J_{412}
&\le C\frac{\e^2}{\l}\Dcal_p(U)\int_\RR a' \int_{\x_0}^\x Q(v|\vtil) d\z d\x
+C\frac{\e^6}{\l^4}\Gcal_p(U)\int_\RR a' \one{p(v)-p(\vtil)>\d_3} d\x\\
&\le C\frac{\e^3}{\l^2}\Dcal_p(U)+C\frac{\e^8}{\l^5}\Gcal_p(U).
\end{align*}
For \(J_{413}\), we apply \eqref{pwh}, Lemma \ref{lemma_pushing}, \eqref{locE}, \eqref{Qconst} and \eqref{Q-lin} to obtain
\begin{align*}
&J_{413}
\le C\frac{\e^2}{\l}\Dcal_h(U)\int_\RR a' \int_{\x_0}^\x v^{\g-1}\one{|h-\htil|>\d_3}d\z d\x
+C\frac{\e^2}{\l}\Dcal_p(U)\int_\RR a' \int_{\x_0}^\x v^{2\g-2}(h-\hbar)d\z d\x\\
&\qquad
+C\frac{\e^6}{\l^4} \int_\RR a' \one{p(v)-p(\vtil)>\d_3} d\x \cdot \int_\RR a'(h-\htil)^2 d\x
+C\frac{\e^6}{\l^4} \Gcal_p(U) \int_\RR a' \one{p(v)-p(\vtil)>\d_3} d\x\\
&\le C\frac{\e}{\l}\Dcal_h(U)\int_\RR a' v^{\g-1}\one{|h-\htil|>\d_3} d\x
+C\frac{\e}{\l}\Dcal_p(U)\int_\RR a' v^{2\g-2}(h-\hbar) d\x
+C\frac{\e^8}{\l^5} \Gcal_p(U)\\
&\le C\frac{\e}{\l}\Dcal_h(U) \int_\RR (Q(v|\vtil)+|h-\htil|^2) d\x
+C\frac{\e}{\l}\Dcal_p(U) \int_\RR a' v^{4\g-4} \one{|h-\htil|>\d_3} d\x
+C\frac{\e^8}{\l^5} \Gcal_p(U)\\
&\le C\frac{\e^3}{\l^2}(\Dcal_p(U)+\Dcal_h(U))
+C\frac{\e^8}{\l^5} \Gcal_p(U).
\end{align*}
The condition \(\g \le 5/4\) is essential here, as it guarantees that \(v^{4\g-4}\le CQ(v|\vtil)\) for large \(v\).
Combining all the estimates above, we obtain 
\[
\abs{J_{41}} \le
C\frac{\e^3}{\l^2}(\Dcal_p(U)+\Dcal_h(U))
+C\frac{\e^2}{\l}\Gcal_p(U).
\]
For \(J_{42}\), using Young's inequality, the mean value theorem and \eqref{u-aux} yields that
\begin{align*}
\abs{J_{42}}
&\le C\frac{\e^2}{\l} \int_\RR a'(h-\htil)^2 d\x
+C\frac{\e^2}{\l} \int_\RR a'Q(v|\vtil) d\x \\
&\le C\frac{\e^2}{\l}\big(\Gcal_h^-(U)+\Gcal_h^+(U)+\Gcal_p(U)\big) +C\frac{\e^3}{\l^2}\Dcal_p(U).
\end{align*}
For \(J_{43}\), since \(\frac{1}{v^{\a+1}}\) is bounded on the region \(\{p(v)-p(\vtil)<-\d_3\}\), we apply Young's inequality with \eqref{u-aux} and \eqref{Qconst} to find
\begin{align*}
\abs{J_{43}}
&\le C\frac{\e^2}{\l} \int_\RR a'(h-\htil)^2 d\x
+C\frac{\e^2}{\l} \int_\RR a'Q(v|\vtil) d\x \\
&\le C\frac{\e^2}{\l}\big(\Gcal_h^-(U)+\Gcal_h^+(U)+\Gcal_p(U)\big) +C\frac{\e^3}{\l^2}\Dcal_p(U).
\end{align*}
Combining the estimates of $J_{4i}$ for $i=1,2,3$, we derive \eqref{para-B7}.

\vspace{2mm}
\bpf{para-B8}
First of all, we consider the following decomposition:
\begin{align*}
\abs{B_8(U)}
&\le C\int_\RR a \abs{(h-\htil)v^{\b-1} \rd_\x(p(v)-p(\vtil)) \htil'} d\x\\
&\qquad + C\int_\RR a \abs{(h-\htil)(v^{\b-1}-\vtil^{\b-1}) p(\vtil)' \htil'} d\x
\eqqcolon J_{51}+J_{52}.
\end{align*}
For \(J_{51}\), using Young's inequality, we obtain
\[
\abs{J_{51}} \le C\frac{\e}{\l}\Dcal_p(U) + C\frac{\e}{\l} \int_\RR |a'|^2 v^{\b-2}(h-\htil)^2 d\x.
\]
Then, since \(-(\g+\a) \le \b-2 = \g-\a-2 \le 0\), using \eqref{Linfty-G}, \eqref{locE} and \eqref{u-aux} yields that
\begin{align*}
C\frac{\e}{\l} \int_\RR |a'|^2 v^{\b-2}(h-\htil)^2 d\x
&\le C\frac{\e^5}{\l^3} \Dcal_p(U) + C\e^2 \int_\RR a'(h-\htil)^2 d\x\\
&\le C\frac{\e^3}{\l} \Dcal_p(U) + C\e^2 \big(\Gcal_h^-(U)+\Gcal_h^+(U)+\Gcal_p(U)\big).
\end{align*}
Thus, we have 
\[
\abs{J_{51}} \le C\frac{\e}{\l} \Dcal_p(U) + C\e^2 \big(\Gcal_h^-(U)+\Gcal_h^+(U)+\Gcal_p(U)\big).
\]
For \(J_{52}\), we split its integrand into three pieces:
\begin{align*}
\abs{J_{52}}&\le C\frac{\e^2}{\l^2}\int_\RR |a'|^2 \big|(h-\htil)(v^{\b-1}-\vtil^{\b-1})\big|
\one{p(v)-p(\vtil)>\d_3}d\x\\
&\qquad+ C\frac{\e^2}{\l^2}\int_\RR |a'|^2 \big|(h-\htil)(v^{\b-1}-\vtil^{\b-1})\big|
\one{\abs{p(v)-p(\vtil)}\le\d_3}d\x\\
&\qquad+ C\frac{\e^2}{\l^2}\int_\RR |a'|^2 \big|(h-\htil)(v^{\b-1}-\vtil^{\b-1})\big|
\one{p(v)-p(\vtil)<-\d_3}d\x \eqqcolon J_{521}+J_{522}+J_{523}.
\end{align*}
For \(J_{521}\), since \(-(\g+\a) \le \b-1 \le 0\), we apply \eqref{Linfty-G} to find that
\[
|a'(\x) (v^{\b-1}-\vtil^{\b-1})(\x)| \le C\frac{\e^2}{\l}\Dcal_p(U) + C\e\l.
\]
Thus, using \eqref{locE}, \eqref{Qconst} and \eqref{u-aux}, we have
\begin{align*}
\abs{J_{521}}
&\le C\frac{\e^2}{\l^2} \Big(\frac{\e^2}{\l}\Dcal_p(U) + \e\l\Big)\int_\RR a'(h-\htil)\one{p(v)-p(\vtil)>\d_3}d\x \\
&\le C\frac{\e^2}{\l^2} \Big(\frac{\e^2}{\l}\Dcal_p(U) + \e\l\Big) \int_\RR a'(Q(v|\vtil)+|h-\htil|^2) d\x\\
&\le C\frac{\e^3}{\l} \big(\Gcal_h^-(U)+\Gcal_h^+(U)+\Gcal_p(U)\big)
+C\frac{\e^4}{\l^2} \Dcal_p(U).
\end{align*}
For \(J_{522}\), using Young's inequality and the mean value theorem with \eqref{u-aux}, we have
\begin{align*}
\abs{J_{522}}
&\le C\frac{\e^3}{\l} \int_\RR a'(h-\htil)^2 d\x
+C\frac{\e^3}{\l} \int_\RR a'Q(v|\vtil) d\x \\
&\le C\frac{\e^3}{\l}\big(\Gcal_h^-(U)+\Gcal_h^+(U)+\Gcal_p(U)\big) +C\frac{\e^4}{\l^2}\Dcal_p(U).
\end{align*}
For \(J_{523}\), since \(|v^{\b-1}-\vtil^{\b-1}|\le C\) on \(\{p(v)-p(\vtil)<-\d_3\}\), it follows that
\[
\abs{J_{523}}
\le C\frac{\e^3}{\l} \int_\RR a'(h-\htil)^2 d\x
+C\frac{\e^3}{\l} \int_\RR a'\one{p(v)-p(\vtil)<-\d_3}d\x.
\]
Hence, using \eqref{u-aux} and \eqref{Qconst}, we have 
\[
\abs{J_{523}} \le C\frac{\e^3}{\l}\big(\Gcal_h^-(U)+\Gcal_h^+(U)+\Gcal_p(U)\big) +C\frac{\e^4}{\l^2}\Dcal_p(U).
\]
Thus, summing up all, we now have a bound for \(J_{52}\): 
\[
\abs{J_{52}} \le C\frac{\e^3}{\l} \big(\Gcal_h^-(U)+\Gcal_h^+(U)+\Gcal_p(U)\big) +C\frac{\e^4}{\l^2} \Dcal_p(U).
\]
From the estimates of $J_{51}$ and $J_{52}$, we get the desired result \eqref{para-B8}.

\vspace{2mm}
\bpf{para-2-total}
\eqref{para-2-total} is a direct consequence of \eqref{para-B6}, \eqref{para-B7} and \eqref{para-B8}.

\bpf{para-2-total-0}
\eqref{para-2-total} is a direct consequence of \eqref{para-2-total} and \eqref{locE}.
Although \(\Gcal_h^+\) cannot be bounded directly by \eqref{locE}, the desired bound can be obtained by applying \eqref{locE} in the course of the proof of \eqref{para-B6}--\eqref{para-B8}, prior to the use of \eqref{u-aux}.
\end{proof}

\subsection{Estimates on the shift part} \label{sec:shift}
To use Proposition \ref{prop:main3} in the proof of Proposition \ref{prop:main}, we need to extract \(-\frac{1}{\e\d_3}|\Ical_g^Y(\vbar)|^2\) from \(-\frac{1}{\e^4}|Y(U)|^2\).
For this purpose, we aim to demonstrate that \(|Y(U)-\Ical_g^Y(\vbar)|\) is negligible.
Thus, we recall the definition of \(Y\) in \eqref{ybg-first} and consider the following decomposition:
\[
Y=Y_g+Y_b+Y_l+Y_s,
\]
where
\begin{align*}
Y_g &\coloneqq
-\frac{1}{2\s_\e^2} \int_\O a' (p(v)-p(\vtil))^2 d\x
-\int_\O a' Q(v|\vtil) d\x \\
&\qquad\qquad
-\int_\O a p(\vtil)' (v-\vtil) d\x 
+\frac{1}{\s_\e} \int_\O a \htil' (p(v)-p(\vtil)) d\x, \\
Y_b &\coloneqq
-\frac{1}{2}\int_\O a'\Big(h-\htil-\frac{p(v)-p(\vtil)}{\s_\e}\Big)^2 d\x 
-\frac{1}{\s_\e}\int_\O a'(p(v)-p(\vtil))\Big(h-\htil-\frac{p(v)-p(\vtil)}{\s_\e}\Big) d\x,\\
Y_l &\coloneqq
\int_\O a \htil' \Big(h-\htil-\frac{p(v)-p(\vtil)}{\s_\e}\Big) d\x, \\
Y_s &\coloneqq
-\int_{\O^c} a' Q(v|\vtil) d\x
-\int_{\O^c} a p(\vtil)' (v-\vtil) d\x
-\int_{\O^c} a'\frac{(h-\htil)^2}{2} d\x
+\int_{\O^c} a \htil' (h-\htil) d\x.
\end{align*}
Note that \(Y_g\) consists solely of terms involving \(p(v)-p(\vtil)\) (or \(v-\vtil\)).
In addition, \(Y_b\) is composed of remaining quadratic terms and \(Y_l\) of linear terms, whereas \(Y_s\) comprises the terms corresponding to the region where \(p(v)\) is large (equivalently, where \(v\) is small).
Given this decomposition, we present the following proposition, which establishes that \(Y_g(U)-\Ical_g^Y(\vbar),Y_b(U),Y_l(U)\), and \(Y_s(U)\) are negligible compared to the good terms.

\begin{prop} \label{prop:shift}
For any fixed constant state \(U_-\coloneqq(v_-,u_-)\in\RR^+\times\RR\) and any constant \(C^*>0\), there exist constants \(\e_0,\d_0,C>0\) (in particular, \(C\) depends on \(C^*\)) such that for any \(\e<\e_0\) and \(\d_0^{-1}\e <\l<\d_0<1/2\) and any \(U\in\Hcal_T\) with \(\abs{Y(U)}\le\e^2\) and \(\Dcal_p(U)\le C^*\frac{\e^2}{\l}\), the following holds.
\begin{equation}\label{Y-conc}
\begin{aligned}
&\abs{Y_g(U)-\Ical_g^Y(\vbar)}^2+\abs{Y_b(U)}^2+\abs{Y_l(U)}^2+\abs{Y_s(U)}^2 \\
&\le C\frac{\e^2}{\l}\Big(\frac{\e}{\l}\Dcal_p(U)+\sqrt{\frac{\e}{\l}}\Gcal_p(U)+(\Gcal_p(U)-\Gcal_p(\Ubar))
+\Gcal_h^-(U) +\sqrt{\frac{\l}{\e}}\Gcal_h^+(U)\Big)
\end{aligned}
\end{equation}
\end{prop}
\begin{proof}
Owing to its structural similarity with the barotropic NS system, the shift part \(Y\) coincide with that appearing in the NS case.
The proof therefore proceeds along the same lines as in \cite[Proposition 4.5]{EEKO-ISO}, and the details are omitted.
\end{proof}

\subsection{Proof of Proposition \ref{prop:main}} \label{sec:proof}
In this subsection, we combine all the estimates obtained in the previous subsections to complete the proof of Proposition \ref{prop:main}. The proof is divided into two cases according to the size of \(\Dcal_p(U)\).

\step{1}
We begin by considering the case \(\Dcal_p(U)\ge 6C^*\frac{\e^2}{\l}\), where the constant \(C^*\) is chosen to be the greater one among constants \(C^*\) given in Proposition \ref{prop:hyp}, \ref{prop:para-1} and \ref{prop:para-2}.
Then, applying \eqref{hyp-total}, \eqref{para-1-total}, and \eqref{para-2-total-0}, we take \(\d_0\) sufficiently small so that we obtain
\begin{align*}
\Rcal(U)
&\le
\abs{\Bcal_{\d_3}(U)} + \abs{\Pcal_1(U)} + \abs{\Pcal_2(U)} - (1-\d_0)\Dcal_p(U) - \frac{1}{2}\Dcal_h(U) \\
&\le 3C^*\frac{\e^2}{\l} - (1-C\d_0)\Dcal_p(U) - \big(\frac{1}{2}-C\d_0\big)\Dcal_h(U)
\le 3C^*\frac{\e^2}{\l} - \frac{1}{2}\Dcal_p(U) \le 0.
\end{align*}

\step{2}
We now consider the other alternative, i.e., \(\Dcal_p(U)\le 6C^* \frac{\e^2}{\l}\).
To proceed further, it is necessary to first determine the truncation size \(\d_3>0\).
Thanks to \eqref{lbisC2}, we choose \(\d_3\) of Proposition \ref{prop:main3} corresponding to the constant \(C_2\) in \eqref{lbisC2}.
This, in turn, naturally determines \(\vbar\).
Then, we extract \(-|\Ical_g^Y(\vbar)|^2\) from \(-|Y(U)|^2\) as follows:
\[
-5\abs{Y(U)}^2 \le -\abs{\Ical_g^Y(\vbar)}^2 + 5\abs{Y_g(U)-\Ical_g^Y(\vbar)}^2 + 5\abs{Y_b(U)}^2 + 5\abs{Y_l(U)}^2 + 5\abs{Y_s(U)}^2.
\]
Now, taking \(\d_0>0\) sufficiently small such that \(\d_0 \le \d_3^5\), we then use \eqref{keyD} to obtain that for \(\e<\e_0(\le \d_3)\) and \(\e/\l<\d_0\), 
\begin{align*}
\Rcal(U) &\le -\frac{5}{\e\d_3}Y(U)^2 + \Bcal_{\d_3}(U) + \Pcal_1(U) + \Pcal_2(U) \\
&\qquad -\Gcal_h^-(U)-\Gcal_h^+(U) -\left(1-\d_0\frac{\e}{\l}\right)\Gcal_p(U) - \frac{1}{2}\Dcal_h(U) - (1-\d_0)\Dcal_p(U) \\
&\le -\frac{1}{\e\d_3}\abs{\Ical_g^Y(\vbar)}^2 + \Ical_1(\vbar) + \Ical_2(\vbar) 
-\left(1-\d_3\frac{\e}{\l}\right)\Gcal_p(\Ubar) - (1-\d_3)\Dcal_p(\Ubar) \\
&\qquad
+\underbrace{
\big(|B_1^+(U)-\Ical_1(\vbar)| + |B_1^-(U)| + |B_2(U)-\Ical_2(\vbar)|\big)}_{\eqqcolon J_{61}}
+\underbrace{|\Pcal_1(U)|+|\Pcal_2(U)|}_{\eqqcolon J_{62}}\\
&\qquad
+\underbrace{\frac{5}{\e\d_3}\big(|Y_g(U)-\Ical_g^Y(\vbar)|^2 + |Y_b(U)|^2 + |Y_l(U)|^2 + |Y_s(U)|^2\big)}_{\eqqcolon J_{63}}
-\Gcal_h^-(U)-\Gcal_h^+(U)\\
&\qquad
-\frac{1}{2}(\Gcal_p(U)-\Gcal_p(\Ubar))
- (\d_3-\d_0)\frac{\e}{\l}\Gcal_p(\Ubar) 
- (\d_3-\d_0)\Dcal_p(U) - \frac{1}{2}\Dcal_h(U).
\end{align*}
Thanks to Proposition \ref{prop:main3}, the first line of the right-hand side, which is $\Rcal_{\e,\d_3}$, is non-positive.
Accordingly, it remains to show that the remaining good terms are large enough to control \(J_{61},J_{62}\) and \(J_{63}\).\\
For \(J_{61}\), using \eqref{bo1p}--\eqref{bo2} and taking \(\d_0(\le \d_3^3)\) small enough, it holds that for any \(\e,\l\) with \(\l<\d_0\) and \(\e/\l < \d_0\),
\begin{align*}
J_{61} &\le \d_3 \Gcal_h^-(U) + C\frac{\e}{\l}\Dcal_p(U)+ C\frac{\e}{\l}(\Gcal_p(U)-\Gcal_p(\Ubar))+ C\frac{\e^4}{\l^2} \Gcal_p(U)  \\
&\le \d_3\Gcal_h^-(U) + \frac{\d_3}{4}\Dcal_p(U)
+ \frac{\d_3}{4}(\Gcal_p(U)-\Gcal_p(\Ubar))
+ \frac{\d_3}{4}\frac{\e}{\l}\Gcal_p(\Ubar).
\end{align*}
For \(J_{62}\), using \eqref{para-1-total} and \eqref{para-2-total}, it also holds that
\begin{align*}
J_{62} &\le C \Big(\l+\frac{\e}{\l}\Big)\Dcal_p(U) + \l\Dcal_h(U)
+ C\e\Big(\Gcal_h^-(U)+\Gcal_h^+(U)+\Gcal_p(U)\Big)\\
&\le \frac{\d_3}{4} \Dcal_p(U) + \frac{1}{2} \Dcal_h(U)
+ \frac{\d_3}{4}(\Gcal_p(U)-\Gcal_p(\Ubar))
+ \frac{\d_3}{4}\frac{\e}{\l}\Gcal_p(\Ubar).
\end{align*}
For \(J_{63}\), using \eqref{Y-conc}, we obtain
\begin{align*}
J_{63} 
&\le \frac{C}{\d_3}\sqrt{\frac{\e}{\l}}
\Big(\Dcal_p(U)
+ \Gcal_h^-(U) + \Gcal_h^+(U)
+(\Gcal_p(U)-\Gcal_p(\Ubar))
+\frac{\e}{\l}\Gcal_p(\Ubar)\Big) \\
&\le \frac{\d_3}{4}\Big(\Dcal_p(U) + \Gcal_h^-(U)+\Gcal_h^+(U)
+(\Gcal_p(U)-\Gcal_p(\Ubar)) 
+\frac{\e}{\l}\Gcal_p(\Ubar)\Big).
\end{align*}
Combining the estimates for $J_{6i}$, we get \eqref{contraction}.

It remains to prove \eqref{bad-bound2}, which can be established without the need for the case division.
Gathering \eqref{hyp-total}, \eqref{para-1-total} and \eqref{para-2-total-0}, we find that
\[
|\Bcal_{\d_3}(U)| + |\Pcal_1(U)| + |\Pcal_2(U)|
\le C \Big(\l+\frac{\e}{\l}\Big)\Dcal_p(U) + \l\Dcal_h(U) + C\frac{\e^2}{\l}.
\]
Then, we recall the definition of \(\Bcal_{\d_3}\) and \(\Jcal^{bad}\) and thus we obtain 
\[
|\Jcal^{bad}(U)| \le |\Bcal_{\d_3}(U)|
+ C\abs{\int_\O a'(h-\htil)^2 d\x}
+ C\abs{\int_\O a'(p(v)-p(\vtil))^2d\x}
\le |\Bcal_{\d_3}(U)| + C\l.
\]
Here, we used \eqref{locE} and \eqref{totvar-ai}.
This completes the proof of \eqref{bad-bound2} and Proposition \ref{prop:main}. \qed

\section{Proof of Theorem \ref{thm:existence}}\label{sec:5}
\setcounter{equation}{0}

In this section, we provide the proof of Theorem \ref{thm:existence}.
The central idea is to exploit the contraction property established in Theorem \ref{thm_main} (or Theorem \ref{thm_main3}) as an a priori estimate.
More precisely, due to the structural features of the NSK system, a direct application of the conventional \(L^2\)-energy estimate is not feasible.
Instead, we employ the contraction property as a substitute for such an \(L^2\)-energy estimate.
It should be noted that, once uniform lower and upper bounds on \(v\) are established, the contraction property effectively reduces to an \(L^2\)-energy estimate by the (locally) quadratic nature of the relative entropy as in Remark \ref{rmk:rel}.
To establish global-in-time solutions, we make use of the standard continuation argument---namely, combining local existence with an a priori estimate to demonstrate that the maximal lifetime of the solutions is infinite.
Since the contraction property is valid only for solutions belonging to the functional class \(\Xcal_T\), appropriate regularity must be taken into account both in the local existence result and in the a priori estimates.

On the other hand, due to a technical limitation, our existence result only covers more restrictive setting than the one required for establishing the contraction property in condition \eqref{g-a-condition}.
In particular, the additional assumption \(\a<1/2\) is necessary.
Moreover, since the constants \(\e_0\) and \(\d_0\) required for the contraction property have already been determined in Theorem \ref{thm_main}, we may fix the two end states accordingly.
As a result, we shall write \(\s\) in place of \(\s_\e\).

\subsection{Main propositions}
In this subsection, we present two main propositions and the proof of Theorem \ref{thm:existence}.
We begin by establishing the local-in-time existence result.

\begin{prop} \label{prop:lwp}
Let \((v_\pm,u_\pm)\in\RR^+\times\RR\) be two given constant states which satisfy \eqref{end-con}.
Let \(\Util\coloneqq(\vtil,\util)\) be the traveling wave solution of \eqref{shock_0}, which connects \((v_-,u_-)\) and \((v_+,u_+)\), with the traveling speed \(\s\).
Then, for any \(M_0\) and \(r_0\), there exists \(\That>0\) such that the following holds.
For any initial datum \(U_0=(v_0,u_0)\) which satisfies 
\[
\|v_0-\vtil\|_{H^2(\RR)}+\|u_0-\util\|_{H^1(\RR)} \le M_0, \quad \text{and} \quad
\inf_{x\in\RR} v_0 \ge r_0,
\]
there exists a unique solution \(U=(v,u)\) to \eqref{main0} on \([0,\That]\) with the initial datum \((v_0,u_0)\) such that
\begin{align*}
&(v-\vtil,u-\util) \in C([0,T];H^2(\RR))\cap L^2(0,T;H^3(\RR)) \times C([0,T];H^1(\RR))\cap L^2(0,T;H^2(\RR)),\\
&\qquad\sup_{t\in[0,\That]}\|v-\vtil\|_{H^2(\RR)}+\sup_{t\in[0,\That]}\|u-\util\|_{H^1(\RR)} \le 2M_0, \quad \text{and} \quad
\inf_{t\in[0,\That]}\inf_{x\in\RR} v \ge \frac{r_0}{2}.
\end{align*}
\end{prop}

\begin{proof}
The proof of local well-posedness follows from the standard approach of constructing a sequence of approximate solutions combined with a Cauchy-type estimate (see for instance, \cite{Solonnikov}).
For the sake of brevity, we do not provide the detailed argument here.
\end{proof}

To extend the solution globally, we then formulate the main proposition concerning an a priori uniform estimates.

\begin{prop} \label{prop:apriori}
Under the same assumptions in Theorem \ref{thm:existence}, if \(U\) is a solution of \eqref{main0} on \([0,T_0)\) for some \(T_0>0\) such that
\begin{align*}
&(v-\vtil,u-\util) \in C([0,T];H^2(\RR))\cap L^2(0,T;H^3(\RR))
\times C([0,T];H^1(\RR))\cap L^2(0,T;H^2(\RR)),\\
&\hspace{39mm} 0 < v^{-1} \in L^\infty(0,T;L^\infty(\RR)), \,\,\,\,\, \forall\, T\in(0,T_0),
\end{align*}
then there exists a constant \(C(T_0)\) such that 
\begin{align*}
&\sup_{t\in[0,T_0)}\|v-\vtil\|_{H^2(\RR)}
+\sup_{t\in[0,T_0)}\|u-\util\|_{H^1(\RR)}
+\sup_{t\in[0,T_0)} \|1/v\|_{L^\infty(\RR)} \le C(T_0).
\end{align*}
\end{prop}


The combination of these two propositions allows us to establish the global-in-time existence via a standard continuation argument.


Suppose that a global-in-time solution does not exist and there exists the finite maximal lifetime \(T_0\) for some \(T_0 \in (\That,\infty)\).
Then, at least one of the following must hold:
\[
\sup_{t\in[0,T_0)} \|v-\vtil\|_{H^2(\RR)} = \infty, \qquad
\sup_{t\in[0,T_0)} \|u-\util\|_{H^1(\RR)} = \infty, \qquad
\inf_{(0,T_0)} \inf_\RR v =0.
\]
However, Proposition \ref{prop:apriori} guarantees that
\[
\sup_{t\in[0,T_0)} \|v-\vtil\|_{H^2(\RR)}
+\sup_{t\in[0,T_0)} \|u-\util\|_{H^1(\RR)}
+\sup_{t\in[0,T_0)} \|1/v\|_{L^\infty(\RR)} \le C(T_0),
\]
for some constant \(C(T_0)\) which is independent of \(T<T_0\).
This leads to a contradiction.
Therefore, a global solution is obtained, and its uniqueness follows from a standard energy method; see for example, \cite[Appendix B]{CKV-JMPA20}.
This completes the proof of Theorem \ref{thm:existence}. \qed

\vspace{2mm}
Thus, it remains to prove Proposition \ref{prop:apriori}

\subsection{Proof of Proposition \ref{prop:apriori}}
First of all, we slightly abuse notation and write \(U=(v,h)\), instead of \((v,u)\), throughout the remainder of this section.
Under this notation, by the definitions of \(\Xcal_T\) and \(\Hcal_T\), the local solution \(U=(v,h)\) belongs to the function class \(\Hcal_T\), for any \(T\in(0,T_0)\).

\vspace{2mm}
\(\bullet\) \textbf{Notation:} In what follows, \(C\) denotes a positive constant which may vary from line to line, and depends on the initial data (for instance, \(\|v_0-\vtil\|_{H^2}\), \(\|1/v\|_{L^\infty}\)), the end states \((v_\pm,u_\pm)\) and \(T_0\), but does not depend on \(T\in(0,T_0)\).
It may also depend on \(\e,\l,\e_0\) and \(\d_0\). We also recall the notation \(f^{\pm X}\coloneqq f(t,x\pm X(t))\) and extend it to arbitrary functions \(X\).

\vspace{2mm}
We restate the contraction inequality \eqref{cont_main} in the following way.
Since \(v_0\) and \(\vtil\) are both bounded above and below away from zero, it holds by Remark \ref{rmk:rel} that
\[
\int_\RR \et(U_0(x)|\Util(x))dx \le C \|U_0-\Util\|_{L^2(\RR)}^2
\le C.
\]
Then, thanks to the boundedness of the weight function \(a\), \eqref{cont_main} and \eqref{shift-control} implies
\begin{equation} \label{from_cont}
\begin{aligned}
&\int_\RR \et(U(t,x+X(t))|\Util(x-\s t))dx \\
&\qquad
+\int_0^T \int_\RR \abs{ a'(x-\s t)} Q(v(s,x+X(s))|\vtil(x-\s t))dx ds \\
&\qquad
+\int_0^T \int_\RR v^\b(s,x+X(s)) \big| \rd_x \big(p(v(s,x+X(s)))-p(\vtil(x-\s t))\big) \big|^2 dx ds \\
&\qquad
+\int_0^T \int_\RR v^{-\a-1}(s,x+X(s)) \big|\rd_x \big(h(s,x+X(s))-\htil(x-\s t)\big)\big|^2 dx ds \\
&\le C \int_\RR \et(U_0(x)|\Util(x))dx \le C,
\end{aligned}
\end{equation}
and
\begin{equation} \label{shift-bound}
\begin{aligned}
&|\dot{X}(t)+\s| \le C(f(t)+1), \quad \text{for a.e. } t\in[0,T], \\
&\text{for some positive function } f \text{ satisfying } \norm{f}_{L^1(0,T)} \le C\int_\RR \et(U_0(x)|\Util(x))dx\le C.
\end{aligned}
\end{equation}
Then, using \eqref{shift-bound}, we find that for any \(t\in[0,T]\),
\begin{equation} \label{shift-bound2}
|X(t)| \le C \Big(\int_\RR \et(U_0(x)|\Util(x))dx + 1\Big) (t+1)
\le C(T_0+1) \le C.
\end{equation}

\subsubsection{Uniform bound on the relative entropy}
We now demonstrate that the relative entropy is finite up to any finite time even without any shift functions.
To show this, we use \eqref{from_cont} and \eqref{shift-bound2} to obtain
\begin{align*}
&\int_\RR (h-\htil^{-\s t})^2 dx
\le 2\int_\RR (h(t,x)-\htil^{-X(t)-\s t})^2 dx
+2\int_\RR (\htil^{-X(t)-\s t}-\htil^{-\s t})^2 dx\\
&\qquad=2\int_\RR (h(t,x+X(t))-\htil^{-\s t})^2 dx
+2\int_\RR (\htil^{-X(t)}-\htil)^2 dx \le C.
\end{align*}
Next, we deal with the term \(Q(v|\vtil)\) in the relative entropy.
We first utilize the definition of \(Q(\cdot|\cdot)\) and consider the following decomposition:
\begin{align*}
\int_\RR Q(v|\vtil^{-\s t})dx
&=\int_\RR Q(v|\vtil^{-X(t)-\s t}) dx
+\int_\RR Q(\vtil^{-X(t)-\s t}|\vtil^{-\s t}) dx\\
&\qquad
+\int_\RR (Q'(\vtil^{-\s t})-Q'(\vtil^{-X(t)-\s t}))(v-\vtil^{-X(t)-\s t}) dx.
\end{align*}
We simply apply \eqref{from_cont} for the first term on the right-hand side and use Remark \ref{rmk:rel} and \eqref{shift-bound2} for the second term.
For the last one, since it holds by the same rationale as \eqref{Q-lin} that \(\big|v-\vtil^{-X(t)-\s t}\big| \le C Q(v|\vtil^{-X(t)-\s t})+ C\), we use \eqref{from_cont} and \eqref{shift-bound2} to find that
\begin{align*}
&\int_\RR \abs{Q'(\vtil^{-\s t})-Q'(\vtil^{-X(t)-\s t})}
\big|v-\vtil^{-X(t)-\s t}\big| dx\\
&\qquad\le C\int_\RR Q(v|\vtil^{-X(t)-\s t}) dx
+C \int_\RR \big|\vtil^{-\s t}-\vtil^{-X(t)-\s t}\big| dx
\le C.
\end{align*}
Thus, gathering all, we obtain 
\begin{equation} \label{apriori-rel}
\int_\RR \et(U|\Util^{-\s t})dx
= \int_\RR Q(v|\vtil^{-\s t})dx + \frac{1}{2}\int_\RR (h-\htil^{-\s t})^2 dx \le C.
\end{equation}
This indicates that, even without shift function, the relative entropy functional remains finite for any finite time.

\vspace{2mm}
However, our objective is to establish strong solutions in \(\Xcal_T\).
To this end, it is necessary to control not only the \(L^2\)-norm but also the \(H^2\)-norm of the \(v\) variable and \(H^1\)-norm of the \(h\) variable.
The proof is structured as follows.
First of all, we eliminate the shift dependence in the dissipation terms of \eqref{from_cont}, in the same manner as we did for the relative entropy.
Based on this, we derive a uniform bound for \(v_x\), and then we establish \(L^\infty\) bounds for both \(v\) and \(1/v\).
Then, in view of Remark \ref{rmk:rel}, these bounds, together with the finiteness of the relative entropy functional, imply the finiteness of the \(L^2\)-norm.
Finally, combining all these ingredients, we do the energy methods to obtain the remaining estimates.

\subsubsection{Uniform bounds on \(\|v^{\frac{\b}{2}}p(v)_x\|_{L^2}\) and \(\|v^{-\frac{\a}{2}-\frac{1}{2}}h_x\|_{L^2}\)}
Here, we will use \eqref{from_cont} to show
\begin{equation} \label{apriori-diff}
\int_0^T \int_\RR v^\b(s,x) \big| \rd_x p(v(s,x)) \big|^2
+ \frac{1}{v^{\a+1}(s,x)} \big|\rd_x h(s,x)\big|^2 dx ds \le C.
\end{equation}
To this end, we first note by \eqref{from_cont} that
\begin{align*}
&\int_0^T \int_\RR v^\b(s,x) \big| \rd_x \big(p(v(s,x))-p(\vtil^{-X(t)-\s t})\big) \big|^2 dx ds \\
&\qquad
+\int_0^T \int_\RR \frac{1}{v^{\a+1}(s,x)} \big|\rd_x \big(h(s,x)-\htil^{-X(t)-\s t}\big)\big|^2 dx ds
\le C.
\end{align*}
Thus, it is enough to show that 
\[
\int_0^T \int_\RR v^\b(s,x) \big| \rd_x p(\vtil^{-X(t)-\s t}) \big|^2 dx ds
+\int_0^T \int_\RR \frac{1}{v^{\a+1}(s,x)} \big|\rd_x \htil^{-X(t)-\s t}\big|^2 dx ds
\le C.
\]
For the first term on the left-hand side, since \(v^\b \le C(1+ v) \le C(1+Q(v|\vtil^{-X(t)-\s t}))\),
it follows from \eqref{from_cont} that
\begin{align*}
&\int_0^T \int_\RR v^\b(s,x) \big| \rd_x p(\vtil^{-X(t)-\s t}) \big|^2 dx ds
\le C\int_0^T \int_\RR v^\b(s,x) \big| (\vtil')^{-X(t)-\s t} \big|^2 dx ds\\
&\qquad\qquad
\le C\int_0^T \int_\RR \big(1+Q(v|\vtil^{-X(t)-\s t})\big) \big| (\vtil')^{-X(t)-\s t} \big|^2 dx ds
\le C.
\end{align*}
The second term requires a more delicate analysis.
First, we present the following lemma.
\begin{lemma} \label{lem:refpt}
For any fixed \(t\in[0,T]\), and any \(N > \max \big(2v_\pm,\frac{2}{v_\pm}\big)\), there exist \(y_1,y_2\in\RR\) and a constant \(C>0\) independent of $T$, such that
\begin{equation} \label{refpt}
\big(v-N\big)_+ (t, y_1) = 0, \qquad \big(\frac{1}{v}-N\big)_+ (t, y_2) =0, \qquad
\abs{y_1}+\abs{y_2}\le C.
\end{equation}
\end{lemma}
\begin{proof}
Since the proofs for $(v-N)_+$ and $(\frac1v-N)_+$ are essentially the same, and thus we only prove the case of \((v-N)_+\). Since $v\to v_{\pm}$ as $x\to\pm\infty$, the existence of the point $y$ such that $(v-N)_+(t,y)=0$ is guaranteed by the choice of $N$. Assume that \(y\) is the closest point from the origin that satisfies \eqref{refpt}. This immediately implies that \(Q(v|\vtil^{-\s t})\ge Q(2v_+|v_-)>0\) or \(Q(v|\vtil^{-\s t})\ge Q(2v_-|v_+)>0\) on \([-\abs{y},\abs{y}]\).
Without loss of generality, we may suppose that \(Q(v|\vtil^{-\s t})\ge Q(2v_+|v_-)>0\).
Then, from \eqref{apriori-rel}, we have 
\[
C \ge \int_\RR Q(v|\vtil^{-\s t}) dx
\ge \int_{-\abs{y}}^{\abs{y}} Q(2v_+|v_-) dx
\ge 2\abs{y}Q(2v_+|v_-).
\]
This gives a $T$-independent upper bound of $y_1$ as we desired.
\end{proof}

To exploit Lemma \ref{lem:refpt}, we first observe 
\begin{equation} \label{apriori-diff1}
\int_0^T \int_\RR v^{-\a-\g-2} |v_x|^2 dx ds
\le C\int_0^T \int_\RR v^\b |p(v)_x|^2 dx ds \le C.
\end{equation}
Then, for \(N > \max \big(2v_\pm,\frac{2}{v_\pm}\big)\), we use Lemma \ref{lem:refpt} to find that for any \(z \in \RR\),
\begin{equation} \label{vap1}
\begin{aligned}
\big(v^{-\frac{\a+1}{2}}-N^{{\frac{\a+1}{2}}}\big)_+ (t,z)
&\le \big(v^{-\frac{\a+1}{2}}-N^{{\frac{\a+1}{2}}}\big)_+ (t,y_2)
+C\Big|\int_{y_2}^z v^{-\frac{\a+3}{2}} (v_x) \one{v^{-1} \ge N} dx\Big|\\
&\le C\sqrt{\int_\RR v^{-\a-\g-2} (v_x)^2 dx} \sqrt{\int_\RR v^{\g-1}\one{v^{-1} \ge N} dx}.
\end{aligned}
\end{equation}
Since it holds by the same principle as \eqref{Qconst} that \(v^{\g-1}\one{v^{-1} \ge N} \le Q(v|\vtil^{-\s t})\), \eqref{from_cont} yields that for any \(z\in\RR\)
\begin{align*}
\big(v^{-\frac{\a+1}{2}}-N^{{\frac{\a+1}{2}}}\big)_+ (t,z)
\le C\sqrt{\int_\RR v^{-\a-\g-2} (v_x)^2 dx},
\end{align*}
which together with \eqref{apriori-diff1} implies that \(\|v^{-\frac{\a+1}{2}}\|_{L^2(0,T; L^\infty(\RR))}\le C\).
Hence, we have 
\[
\int_0^T \int_\RR \frac{1}{v^{\a+1}(s,x)} \big|\rd_x \htil^{-X(t)-\s t}\big|^2 dx ds
\le C \int_0^T \|v^{-\frac{\a+1}{2}}\|_{L^\infty(\RR)}^2 \int_\RR \big|\rd_x \htil^{-X(t)-\s t}\big| dx ds
\le C.
\]
This establishes \eqref{apriori-diff} as we desired.

\subsubsection{Uniform bounds on \(\|\frac{\m(v)}{v}v_x\|_{L^\infty L^2}\) and \(\|\frac{\m(v)}{v}\big(\frac{\m(v)}{v}v_x\big)_x\|_{L^2L^2}\)}
We now perform the standard energy estimate for the derivative of \(v\) as follows.
Using \eqref{main00}\(_1\), we have
\begin{align*}
&\frac{d}{dt}\int_\RR \frac{1}{2}\Big(\frac{\m(v)}{v}v_x\Big)^2 dx
+\t_1 \int_\RR \frac{\m(v)}{v}\Big(\Big(\frac{\m(v)}{v}v_x\Big)_x\Big)^2 dx
= - \int_\RR \Big(\frac{\m(v)}{v}v_x\Big)_x \m(v)\frac{h_x}{v} dx \\
&\hspace{10mm}
\le \frac{\t_1}{2} \int_\RR \frac{\m(v)}{v}\Big(\Big(\frac{\m(v)}{v}v_x\Big)_x\Big)^2 dx
+ C\int_\RR \frac{\m(v)}{v}(h_x)^2 dx.
\end{align*}
Thanks to \eqref{apriori-diff}, the second term on the right-hand side is in \(L^1(0,T)\), and thus Gr\"onwall's lemma implies that
\begin{equation} \label{apriori-vH1}
\sup_{t\in[0,T]} \int_\RR \Big(\frac{\m(v)}{v}v_x\Big)^2 dx
+\int_0^T \int_\RR \frac{\m(v)}{v}\Big(\Big(\frac{\m(v)}{v}v_x\Big)_x\Big)^2 dx ds \le C.
\end{equation}

\subsubsection{Uniform bounds on \(\|1/v\|_{L^\infty}\) and \(\|v\|_{L^\infty}\)}
Now we establish \(L^\infty(0,T;L^\infty(\RR))\) bounds for \(1/v\) and \(v\).
To handle \(1/v\), as in \eqref{vap1} with the exponent replaced by \(-\a\), we use \eqref{apriori-rel} and \eqref{apriori-vH1} to find that for any \(z\in\RR\)
\begin{align*}
&\big(v^{-\a}-N^{\a}\big)_+ (t,z)
\le \big(v^{-\a}-N^{\a}\big)_+ (t,y_2)
+C\Big|\int_{y_2}^z \frac{v_x}{v^{\a+1}} \one{v^{-1} \ge N} dx\Big|\\
&\le C\sqrt{\int_\RR \Big(\frac{\m(v)}{v}v_x\Big)^2 dx} \sqrt{\int_\RR \one{v^{-1} \ge N} dx}
\le C\sqrt{\int_\RR Q(v|\vtil^{-\s t}) dx}
\le C.
\end{align*}
Note that the constant \(C\) on the right-hand side is independent of \(t\in[0,T]\), and hence this gives an \(L^\infty((0,T)\times\RR)\) bound for \(1/v\).

\vspace{2mm}
For an \(L^\infty((0,T)\times\RR)\) bound for \(v\), we first choose \(k>0\) so that \(2k+2\a\le1\). 
In this part of the argument, the condition \(a<\frac{1}{2}\) becomes crucial.
Then, for any \(z\in\RR\),
\begin{align*}
(v^k-N^k)_+(t,z)
&\le (v^k-N^k)_+(t,y_1)
+C\Big|\int_{y_1}^z v^{k-1}v_x \one{v\ge N} dx\Big|\\
&\le 
C\sqrt{\int_\RR \Big(\frac{\m(v)}{v}v_x\Big)^2 dx} \sqrt{\int_\RR v^{2k+2\a} \one{v \ge N} dx}.
\end{align*}
Then, since there exists a constant \(C>0\) such that \(v^{2k+2\a} \one{v \ge N} \le CQ(v|\vtil^{-\s t})\), we apply \eqref{apriori-rel} and \eqref{apriori-vH1} to establish an \(L^\infty((0,T)\times\RR)\) bound for \(v\).
Thus, in summary, we have
\begin{equation} \label{apriori-Li}
\norm{v}_{L^\infty((0,T)\times\RR)} + \norm{1/v}_{L^\infty((0,T)\times\RR)} \le C.
\end{equation}
Moreover, using \eqref{apriori-rel} with Remark \ref{rmk:rel}, this implies that 
\begin{equation} \label{apriori-LiL2}
\|U-\Util\|_{L^\infty(0,T;L^2(\RR))} \le C.
\end{equation}

\subsubsection{Uniform bounds on the higher order derivatives}
Lastly, we do the energy estimate to obtain control on the higher order derivatives.
To this end, we present the following lemma.
\begin{lemma} \label{lem:refpt1}
For any fixed \(t\in[0,T]\)\, there exist \(y\in\RR\) and a constant \(C>0\) such that
\begin{equation*}
\Big(\m(v)\frac{v_x}{v}-1\Big)_+ (t, y) = 0 \quad \text{and} \quad \abs{y} \le C.
\end{equation*}
\end{lemma}
\begin{proof}
Given \eqref{apriori-vH1}, the proof is analogous to that of Lemma \ref{lem:refpt}.
Thus, we omit it.
\end{proof}

To proceed further, using \eqref{apriori-Li} and \eqref{apriori-vH1}, we observe that for any \(z\in\RR\),
\begin{align*}
&\Big(\m(v)\frac{v_x}{v}-1\Big)_+ (t, z)
=\Big(\m(v)\frac{v_x}{v}-1\Big)_+ (t, y)
+\int_y^z \Big(\m(v)\frac{v_x}{v}\Big)_x \one{\m(v)\frac{v_x}{v}\ge1} dx\\
&\le C\sqrt{\int_\RR \frac{\m(v)}{v}\Big(\Big(\frac{\m(v)}{v}v_x\Big)_x\Big)^2 dx}
\sqrt{\int_\RR \one{\m(v)\frac{v_x}{v}\ge1} dx}\\
&\le C\sqrt{\int_\RR \frac{\m(v)}{v}\Big(\Big(\frac{\m(v)}{v}v_x\Big)_x\Big)^2 dx}
\sqrt{\int_\RR \Big(\m(v)\frac{v_x}{v}\Big)^2 dx}
\le C\sqrt{\int_\RR \frac{\m(v)}{v}\Big(\Big(\frac{\m(v)}{v}v_x\Big)_x\Big)^2 dx}.
\end{align*}
This together with \eqref{apriori-vH1} and \eqref{apriori-Li} implies
\begin{equation} \label{vx-L2Li}
\|v_x\|_{L^2(0,T; L^\infty(\RR))}\le C.
\end{equation}

To obtain an estimate on \(h_x\), we use \eqref{main00}\(_2\) to get
\[
\frac{d}{dt}\frac{1}{2}\int_\RR (h_x)^2 dx
+ \t_2\int_\RR \frac{\m(v)}{v}(h_{xx})^2 dx
=\int_\RR p(v)_x h_{xx}dx
-\t_2\int_\RR \Big(\frac{\m(v)}{v}\Big)_x h_x h_{xx}dx.
\]
Then, for the first term on the right-hand side, using \eqref{apriori-vH1} and \eqref{apriori-Li}, we find that
\[
\int_\RR p(v)_x h_{xx}dx
\le \frac{\t_2}{4}\int_\RR \frac{\m(v)}{v}(h_{xx})^2 dx
+C.
\]
Moreover, for the second term, we apply \eqref{apriori-Li} to obtain
\begin{align*}
\abs{-\t_2\int_\RR \Big(\frac{\m(v)}{v}\Big)_x h_x h_{xx}dx}
&\le \frac{\t_2}{4}\int_\RR \frac{\m(v)}{v}(h_{xx})^2 dx
+C\int_\RR (v_x)^2 (h_x)^2 dx\\
&\le \frac{\t_2}{4}\int_\RR \frac{\m(v)}{v}(h_{xx})^2 dx
+C \|v_x\|_{L^\infty(\RR)}^2 \int_\RR (h_x)^2 dx.
\end{align*}
Thus, we have
\begin{align*}
\frac{d}{dt}\frac{1}{2}\int_\RR (h_x)^2 dx
+ \frac{\t_2}{2}\int_\RR \frac{\m(v)}{v}(h_{xx})^2 dx
\le C \|v_x\|_{L^\infty(\RR)}^2 \int_\RR (h_x)^2 dx+C.
\end{align*}
Thanks to \eqref{apriori-Li}, \eqref{vx-L2Li} and Gr\"onwall's lemma, we finally obtain
\begin{equation} \label{apriori-hH1}
\sup_{t\in[0,T]} \int_\RR (h_x)^2 dx
+\int_0^T \int_\RR (h_{xx})^2 dx ds \le C.
\end{equation}

\vspace{2mm}
We are now ready to improve the regularity of \(v\)-variable.
To this end, we do the energy estimate for \(v_{xx}\) as follows: using \eqref{main00}, it follows that
\begin{align*}
&\frac{d}{dt}\frac{1}{2}\int_\RR (v_{xx})^2 dx
+\t_1\int_\RR \frac{\m(v)}{v}(v_{xxx})^2 dx\\
&=-\int_\RR h_{xx} v_{xxx} dx
-\t_1\int_\RR v_x v_{xxx} \Big(\frac{\m(v)}{v}\Big)_{xx} dx
-2\t_1\int_\RR \Big(\frac{\m(v)}{v}\Big)_x v_{xx} v_{xxx} dx.
\end{align*}
We analyze the right-hand side on a term-by-term basis.
For the first term, we utilize \eqref{apriori-Li} to find that
\[
-\int_\RR h_{xx} v_{xxx} dx 
\le \frac{\t_1}{4}\int_\RR \frac{\m(v)}{v}(v_{xxx})^2 dx
+ C\int_\RR (h_{xx})^2 dx.
\]
For the second term, using \eqref{apriori-Li}, we observe
\begin{align*}
-\int_\RR v_x v_{xxx} \Big(\frac{\m(v)}{v}\Big)_{xx} dx
&\le C \int_\RR \abs{v_x v_{xx} v_{xxx}} dx 
+ C \int_\RR \abs{(v_x)^3 v_{xxx}} dx \\
&\le \frac{\t_1}{4}\int_\RR \frac{\m(v)}{v}(v_{xxx})^2 dx
+ C\|v_x\|_{L^\infty(\RR)}^2 \int_\RR (v_x)^2 + (v_{xx})^2 dx.
\end{align*}
For the last term, similar to the second one, it holds that
\[
-2\int_\RR \Big(\frac{\m(v)}{v}\Big)_x v_{xx} v_{xxx} dx
\le \frac{\t_1}{4}\int_\RR \frac{\m(v)}{v}(v_{xxx})^2 dx
+ C\|v_x\|_{L^\infty(\RR)}^2 \int_\RR (v_{xx})^2 dx.
\]
Thus, gathering all, we have 
\[
\frac{d}{dt}\frac{1}{2}\int_\RR (v_{xx})^2 dx
+\frac{\t_1}{4}\int_\RR \frac{\m(v)}{v}(v_{xxx})^2 dx\\
\le C\int_\RR (h_{xx})^2 dx
+C\|v_x\|_{L^\infty(\RR)}^2 \int_\RR (v_x)^2 + (v_{xx})^2 dx.
\]
Then, we apply Gr\"onwall's lemma with \eqref{apriori-vH1}, \eqref{apriori-Li} and \eqref{vx-L2Li} and obtain
\begin{equation} \label{apriori-vH2}
\sup_{t\in[0,T]} \int_\RR (v_{xx})^2 dx
+\int_0^T \int_\RR (v_{xxx})^2 dx ds \le C.
\end{equation}

\vspace{2mm}
Since our analysis is carried out in the \((v,h)\)-variables rather than in \((v,u)\), we need to get an additional regularity to obtain the desired regularity in \((v,u)\)-variables.
Given \eqref{apriori-hH1}, since \(u-h=\tau_1\frac{\m(v)}{v}v_x\), we need to show that \(\big(\frac{\m(v)}{v}v_x\big)_x \in L^\infty(0,T;L^2(\RR))\).
For this purpose, we perform the energy estimate with the aid of \eqref{apriori-Li} as follows: using \eqref{main00}, it holds that
\begin{align*}
&\frac{d}{dt}\frac{1}{2}\int_\RR \Big(\Big(\frac{\m(v)}{v}v_x\Big)_x\Big)^2 dx
+\t_1\int_\RR \frac{\m(v)}{v} \Big(\Big(\frac{\m(v)}{v}v_x\Big)_{xx}\Big)^2 dx
=-\int_\RR \Big(\frac{\m(v)}{v}v_x\Big)_{xx} \Big(\frac{\m(v)}{v}h_x\Big)_x dx\\
&\qquad \le 
\frac{\t_1}{2}\int_\RR \frac{\m(v)}{v} \Big(\Big(\frac{\m(v)}{v}v_x\Big)_{xx}\Big)^2 dx
+C\int_\RR (v_x)^2 (h_x)^2 dx
+C\int_\RR (h_{xx})^2 dx.
\end{align*}
This together with \eqref{vx-L2Li}, \eqref{apriori-hH1} and \eqref{apriori-Li} yields that
\begin{equation} \label{apriori-vH2-2}
\sup_{t\in[0,T]} \int_\RR \Big(\Big(\frac{\m(v)}{v}v_x\Big)_x\Big)^2 dx
+\int_0^T \int_\RR \Big(\Big(\frac{\m(v)}{v}v_x\Big)_{xx}\Big)^2 dx ds \le C.
\end{equation}
Thus, summing up \eqref{apriori-vH1}--\eqref{apriori-LiL2}, \eqref{apriori-hH1}--\eqref{apriori-vH2-2} and
\[
\|\rd_x \vtil^{-\s t}\|_{L^\infty(0,T;H^1(\RR))}
+\|\rd_x \util^{-\s t}\|_{L^\infty(0,T;L^1(\RR))} \le C,
\]
we complete the proof of Proposition \ref{prop:apriori}. \qed

\section*{Statements and Declarations}

\noindent\textbf{Declaration of competing interest.}
The authors declared that they have no conflict of interest to this work.

\vspace{2mm}
\noindent\textbf{Data availability statement.}
We do not analyze or generate any datasets, because our work proceeds within a theoretical and mathematical approach. 




\bibliography{reference}

\end{document}